\documentclass[11pt,a4paper,reqno]{amsart} 

\usepackage{anysize}
\marginsize{3.5cm}{3.5cm}{2.7cm}{2.7cm}

\usepackage{hyperref}
\hypersetup{
    colorlinks=true,       
    citecolor=blue,        
}

\usepackage[english]{babel}
\usepackage{amsmath}
\usepackage{amsfonts,amssymb}
\usepackage{enumerate}
\usepackage{mathrsfs}
\usepackage{amsthm}
\usepackage{tikz}
\usetikzlibrary{arrows}
\usepackage{mathrsfs}
\usepackage{caption}

\theoremstyle{plain}
\newtheorem{theo}{Theorem}[section]
\newtheorem*{theo*}{Theorem}
\newtheorem*{main*}{Main result}
\newtheorem{prop}[theo]{Proposition}
\newtheorem{lemm}[theo]{Lemma}
\newtheorem{coro}[theo]{Corollary}
\newtheorem{defi}[theo]{Definition}
\theoremstyle{definition}
\newtheorem{rema}[theo]{Remark}
\newtheorem*{rema*}{Remark}

\newtheorem*{nota*}{Notation}

\DeclareMathOperator{\diff}{d}

\def\aeta{\arrowvert_{y=\eta}}

\def\be{\begin{equation}}
\def\dsigma{\diff \! \sigma}
\def\dt{\diff \! t}
\def\dx{\diff \! x}
\def\dX{\diff \! X}
\def\dxdt{\dx \dt}
\def\dy{\diff \! y}
\def\dydt{\diff \! y \diff \! t}
\def\dydx{\diff \! y \diff \! x}
\def\dydxdt{\diff \! y \diff \! x \diff \! t}
\def\ee{\end{equation}}
\def\e{\eqref}
\def\intL{\int_0^L}
\def\intT{\int_0^T}
\def\intTL{\int_0^T\int_0^L}
\def\pair{e}

\def\defn{\mathrel{:=}}

\def\Hr{\mathcal{H}}
\def\la{\left\vert}
\def\lA{\left\Vert}
\def\le{\leq}
\def\les{\lesssim}
\def\mez{\frac{1}{2}}
\def\ra{\right\vert}
\def\rA{\right\Vert}
\def\tdm{\frac{3}{2}}

\def\uq{\frac{1}{4}}

\def\xN{\mathbb{N}}
\def\xR{\mathbb{R}}

\def\xZ{\mathbb{Z}}
\def\px{\partial_x}
\def\py{\partial_y}

\def\ST{\Big\arrowvert_0^T}

\title{Stabilization of gravity water waves}
\author[Thomas Alazard]{Thomas Alazard \\ CNRS \& \'Ecole normale sup\'erieure}

\def\notina[#1]#2{\begingroup\def\thefootnote{\fnsymbol{footnote}}\footnote[#1]{#2}\endgroup}

\begin{document}

\begin{abstract}
This paper is devoted to the stabilization of the incompressible Euler equation with free surface. 
We study the damping of two-dimensional gravity waves by an absorbing beach 
where the water-wave energy is dissipated by using the variations of the external pressure. 
\end{abstract}

\maketitle

\section{Introduction}

Many problems\notina[0]{This work is 
partly supported by the grant ``ANA\'E'' ANR-13-BS01-0010-03.} 
in water-wave theory require 
to study the behavior of waves propagating in an unbounded domain, 
like those encountered in the open sea. 
On the other hand, the numerical analysis of the water-wave equations requires to work 
in a bounded domain. This problem appears for the effective modeling of many partial differential equations 
and several methods have been developed to solve it. 
A classical approach consists in truncating the domain by introducing 
an artificial boundary. This is possible provided that one can find some 
special non-reflecting boundary conditions which 
make the artificial boundary (approximatively) invisible to outgoing waves. 
We refer to the extensive surveys by Israeli and Orszag~\cite{Israeli1981}, Tsynkov~\cite{Tsynkov} 
and also to the recent papers  by Abgrall, Carney, Jennings, Karni, Pridge and Rauch~\cite{JKR,JPCKRA} for the study 
of absorbing boundary conditions for the linearized 2D gravity water-wave equations. 
Another method, which is widespread to study wave equations, 
consists in damping outgoing waves in an absorbing zone surrounding 
the computational boundary (see \cite{Israeli1981,Tsynkov,Bodony2006}). 
For the water-wave equations, the idea of using 
the latter method goes back to Le M\'ehaut\'e~\cite{Mehaute1972} in 1972. 
This approach is very important for the analysis of 
the water-wave equations for at least two 
reasons. 
Firstly, it is used in many numerical studies 
(we refer to \cite{CBS1993,Grilli1997,Clement1999,Duclos2001,Bonnefoy2005,Clamond2005,Ducrozet2007} and 
the references there in) as an efficient 
approach to absorb outgoing waves. 
Secondly, the idea of adding an absorbing layer is also useful 
for the experimental study of water waves in wave basins. 
Indeed, think of a rectangular wave basin, having vertical walls, 
equipped with a wave-maker at one extremity. 
The waves generated by the wave-maker will 
be reflected at the opposite side and then will interact 
with the wave produced by the wave-maker. 
Consequently, to simulate experimentally the open sea propagation, 
one has to introduce 
wave absorbers to minimize wave reflection. 

The mathematical study of the damping properties of these absorbers 
corresponds to the mathematical question of the stabilization of the water-wave equations. 
Our goal in this paper is to start the analysis of this problem for the nonlinear water-wave equations. 

There is a huge literature about the absorption of water-wave energy. We refer the reader to 
the literature review by Ouellet and Datta \cite{Ouellet1986} for a description of the energy 
absorbing devices commonly used in 48 wave basins around the world. 
The most popular wave absorbers are passive absorbers. They consist of  
a beach with a mild slope. The principle is that, when arriving to the artificial beach, 
the steepening of the forward face of waves and their subsequent 
overturning dissipates energy. Another widely used strategy is to introduce a porous media 
to absorb the wave energy. 
The mathematical analysis of these absorbing devices raises extremely difficult questions. 
Consequently, 
to stabilize the water-wave equations or to develop numerical absorbing sponges, 
one prefers to use simpler means to dissipate energy. For similar problems, the 
simplest choice could be to use viscous damping, but this is not possible here 
since one considers a potential flow (so that the velocity is harmonic). 
For such a flow, the energy can only be transmitted or dissipated through the free surface. 
This suggests to consider a pneumatic wave maker, that is to say a wave maker where 
the variations of the external pressure acting on the free surface are used to absorb waves. 
This idea goes back to the work by Larsen and Dancy~\cite{Larsen1983}. 
It has been widely used and many elaborations and variants 
have been implemented, in particular by Cl\'ement~\cite{Clement1996} who proposed 
to couple the pneumatic wave-maker with a piston-like absorbing boundary condition at the tank extremity 
(see also \cite{CBS1993,Grilli1997,Clement1999,Duclos2001,Bonnefoy2005,Clamond2005,Ducrozet2007}).

Let us be more specific. Denote by $\mathcal{H}$ the energy of the fluid and by 
$P_{ext}$ the evaluation of the external pressure at the free surface. 
The question is to find an expression 
of $P_{ext}$ in terms of the unknowns such that the following two properties hold:
\begin{enumerate}
\item $P_{ext}$ vanishes away from the artificial beach (also called sponge layer) 
which is the neighborhood of the boundary where one wants to absorb the waves;
\item the energy $\mathcal{H}$ goes to zero (one also wants to determine the rate of decay).
\end{enumerate}
One can easily compute the work done by $P_{ext}$ (see \S\ref{S:HD}) and obtain that
$$
\frac{d\mathcal{H}}{dt}=-\int_{S_{beach}} P_{ext}\, \phi_n  \dsigma,
$$
where $S_{beach}$ is the absorbing zone and $\phi_n$ denotes 
the normal derivative of the velocity potential $\phi$. 
As noted by Cao, Beck and Schultz (\cite{CBS1993}), 
this suggests to set
\be\label{p1}
P_{ext}=\chi \phi_n,
\ee
where $\chi\ge 0$ is a cut-off function. Indeed, with this choice it is obvious that the energy is a non-increasing function. 
The previous observation explains why this choice 
is widespread (see \cite{Grilli1997,Bonnefoy2005,Ducrozet2007}Ê
and the references there in). 

\bigskip

However, to study the stabilization of the water-wave equations, 
the idea of choosing \e{p1} is inapplicable for 
the simple reason that the Cauchy problem seems ill-posed when $P_{ext}$ is given by \e{p1}. 
This question will be studied in a separate paper. Let us only mention that it is a non trivial problem. Indeed, one can modify slightly 
\e{p1} and obtain a system of equations whose Cauchy problem is well-posed. Namely, 
if one replaces the normal derivative $\phi_n$ by the derivative of $\phi$ in the vertical direction, 
then the Cauchy problem is well-posed. 
However, one cannot use the latter choice to stabilize the equations since one cannot prove that the energy is decaying. 

Many other choices for $P_{ext}$ have been used (see for instance the papers by 
Baker, Meiron and Orszag~\cite{Baker1989} and Clamond et al.\ \cite{Clamond2005}) but 
we have not been able to use one of them for the same reasons (either the Cauchy problem is not well-posed 
or one cannot prove that the energy is decaying). 
To overcome this problem, we shall take benefit of an elementary (though seemingly new) observation which
shows that the energy is decaying when $P_{ext}$ satisfies
\be\label{p2}
\partial_x P_{ext}=\chi(x)\int_{-h}^{\eta(t,x)}\phi_x(t,x,y)\dy,
\ee
where $\chi\ge 0$ is a cut-off function, $x$ (resp.\ $y$) is the horizontal (resp.\ vertical) space variable and $\eta$ is the free surface elevation. 
By contrast with \e{p1}, one can easily prove that the Cauchy problem is well-posed when $P_{ext}$ satisfies~\e{p2}. 
In addition, by exploiting several hidden cancellations, we will be able to quantity the decay rate, that is to estimate the ratio 
$\Hr(T)/\Hr(0)$. By assuming that the solution exists on large time interval, this will imply that the energy converges exponentially to zero. 

\bigskip

To conclude this introduction, let us mention that we study only the stabilization problem 
in this paper and we refer to \cite{ABHK,JCHG,Reid1986,Reid1995,ReidRussell1985} 
for the analysis of the generation of water waves in a pneumatic wave maker. 

\subsection*{Organization of the paper}
We gather the statements of our main results in Section~\ref{S:2}. 
Our first main result is an 
integral identity (see Theorem \ref{T1bis}) 
which allows to compare the integral in time of the energy to 
the work done by the external pressure. This identity, which holds for 
any solution and {\em any} external pressure, 
will be proved in Section~\ref{S:5} by adapting the multiplier method to the water-wave problem. 
Since we do not assume that the reader is familiar with control theory, before proving this result 
we will recall in Section~\ref{S:multiplier} some important methods and results. 
We will also explain the main difficulties one has to cope with when adapting these methods to the study of the 
water-wave equations.

As already mentioned, the energy decays when $P_{ext}$ is given either by \e{p1} or \e{p2}. 
In addition, as we will see in Section~\ref{S:smooth}, the Cauchy problem is well-posed when $P_{ext}$ is given by \e{p2}. 
This is why we assume that $P_{ext}$ is given by \e{p2}. 
Our second main result, which is Theorem~\ref{P:quant}, asserts that, by exploiting the integral identity alluded to above, 
one can quantify the decay rate of the energy for small enough solutions. Assuming that the solution exists on large time intervals, 
we will obtain an exponential decay (cf Corollary \ref{Coro:2.4}). This result is stated in Section~\ref{S:2} and proved in Section~\ref{S:6}. 
The latter result holds under a natural assumption about the frequency localization of the solution. 

Eventually, in Appendix~\ref{S:linear} we will prove 
Sobolev estimates for the linearized problem. Also, in Appendix~\ref{A:Identity} we will prove another integral identity, which is not used to prove a stabilization result, but gives an interesting observability inequality.

\bigskip

\subsection*{Acknowledgements}

I gratefully acknowledge 
F\'elicien Bonnefoy and Guillaume Ducrozet for a demonstration of the wave tank of the \'Ecole Centrale de Nantes. 
I would like also to warmly thank Nicolas Burq, Jean-Michel Coron, Emmanuel Dormy and Camille Laurent for stimulating discussions.

\section{Main results}\label{S:2}

\begin{center}
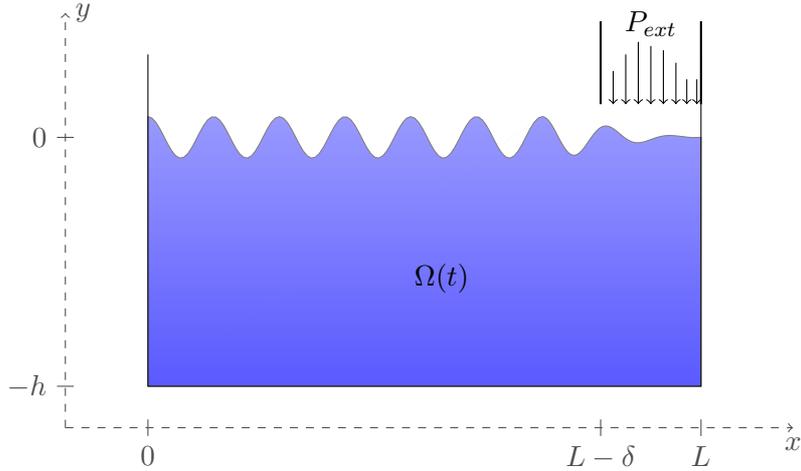
\begin{figure}
\begin{tikzpicture}[scale=1.1,samples=100]
 \shadedraw  [top color=blue!40!white, bottom color=blue!65!white, color=gray] 
(-6.2124,-3) -- (-6.2124,0) --  
plot [domain=0:4.7124,shift={(-6.2124,0)}] ({\x},{0.25*cos(8*\x r)}) 
 plot [domain=0:1.8,shift={(-1.5,0)}] ({\x},{0.25*exp(-\x*\x)*cos(8*\x r)}) -- (0.4,0) -- (0.4,-3) -- (-6.2124,-3)  ;
\draw (-6.2124,1)--(-6.2124,-3) -- (0.4,-3) -- (0.4,1);

\draw [style=thick,shift={(-0.35,0.9)}] (-0.45,0.5) -- (-0.45,-0.5) ;
\draw [style=thin,->,shift={(-0.35,0.9)}] (-0.3,-0.1) -- (-0.3,-0.5);
\draw [style=thin,->,shift={(-0.35,0.9)}] (-0.15,0.1) -- (-0.15,-0.5);
\draw [style=thin,->,shift={(-0.35,0.9)}] (0,0.25) -- (0,-0.5);
\draw [style=thin,->,shift={(-0.35,0.9)}] (0.15,0.2) node [above] {$P_{ext}$} -- (0.15,-0.5);
\draw [style=thin,->,shift={(-0.35,0.9)}] (0.3,0.15)  -- (0.3,-0.5);
\draw [style=thin,->,shift={(-0.35,0.9)}] (0.45,0) -- (0.45,-0.5);
\draw [style=thin,->,shift={(-0.35,0.9)}] (0.58,-0.2) -- (0.58,-0.5);
\draw [style=thin,->,shift={(-0.35,0.9)}] (0.7,-0.2) -- (0.7,-0.5);
\draw [style=thick,shift={(-0.35,0.9)}] (0.75,0.5) -- (0.75,-0.5) ;
\node at (-2.7,-1.7) {$\Omega(t)$};
\draw [color=black!70!white,->,dashed] (-7.2,-3.5) -- (-7.2,1.5) node [right] {$y$};
\draw [color=black!70!white,->,dashed] (-7.2,-3.5) -- (1.5,-3.5) node [below] {$x$};
\draw [color=black!70!white] (-0.8,-3.4) -- (-0.8,-3.6) node [below] {$L-\delta$};
\draw [color=black!70!white] (0.4,-3.4) -- (0.4,-3.6) node [below] {$L$};
\draw [color=black!70!white] (-6.2124,-3.4) -- (-6.2124,-3.6) node [below] {$0$} ;
\draw [color=black!70!white] (-7.3,-3) -- (-7.1,-3) ;
\draw [color=black!70!white] (-7.3,0) -- (-7.1,0) ;
\node [color=black!70!white] at (-7.3,0) [left] {$0$};
\node [color=black!70!white] at (-7.3,-3) [left] {$-h$} ;
\end{tikzpicture}
\caption{Waves generated near $x=0$, propagating to the right, and absorbed in the neighborhood of $x=L$ 
by means of an external counteracting pressure produced by blowing above the free surface.}
\end{figure}
\end{center}

\subsection{The equations} 
We assume that the dynamics is described by the incompressible Euler equations with free surface and consider the 
irrotational case. For the sake of simplicity, we consider a two-dimensional fluid located inside a rectangular tank. 
The water depth is denoted by $h$, the length by $L$ and the free surface elevation by~$\eta$. 
At time~$t$, the fluid domain is 
thus given by
\be\label{b1}
\Omega(t)=\left\{\, (x,y)\,:\, x\in [0,L],~-h\le y\le \eta(t,x)\,\right\},
\ee
where $x$ (resp.\ $y$) is the horizontal (resp.\ vertical) space variable. 

Then the velocity is given by 
$v=\nabla_{x,y}\phi$ for some potential $\phi\colon \Omega\rightarrow \xR$ satisfying
\begin{equation}\label{b5}
\Delta_{x,y}\phi=0,\quad 
\partial_{t} \phi +\mez \la \nabla_{x,y}\phi\ra^2 +P +g y = 0,
\end{equation}
where 
$P\colon \Omega\rightarrow\xR$ is the pressure, 
$g$ is the acceleration of gravity, 
$\nabla_{x,y}=(\partial_x,\partial_y)$ and $\Delta_{x,y}=\partial_x^2+\partial_y^2$. 
Partial differentiation will be denoted by suffixes, so that $\phi_x=\partial_x \phi$ and $\phi_y=\partial_y \phi$ 
(except for $\px P_{ext}$). 
Furthermore, the velocity 
satisfies the solid wall boundary condition on the bottom and the vertical walls, which implies that
\begin{alignat}{3}
&\phi_x =0 \quad&&\text{for}\quad && x=0 \text{ or }x=L,\label{b6}\\
&\phi_y =0 \quad&&\text{for}\quad && y=-h.\label{b7}
\end{alignat}
The problem is then determined by two boundary conditions on the free surface. The first equation 
asserts that the free surface moves with the fluid:
\be\label{b8}
\partial_{t} \eta = \sqrt{1+\eta_x^2}\, \phi_n \arrowvert_{y=\eta}=\phi_y(t,x,\eta)-
\eta_x(t,x)\phi_x(t,x,\eta).
\ee
The second equation is a balance of forces across the free surface. It reads
\be\label{b9}
P\arrowvert_{y=\eta}=P_{ext},
\ee
where $P_{ext}=P_{ext}(t,x)$ is the evaluation of the external pressure at the free surface.

Also we always assume (without explicitly recalling this condition below) that
\be\label{b9b}
\eta\ge -\frac{h}{2},\quad \int_0^L\eta(t,x) \dx=0 \text{ for all time }t.
\ee
One can assume that the mean value of $\eta$ vanishes since it is a conserved quantity. We also
assume that the free surface intersects the vertical walls\footnote{When $P_{ext}=0$, it is proved in \cite{ABZ4} 
that \e{b9a}Ê
always holds for smooth enough solutions. In fact the analysis in \cite{ABZ4} is written only for 
the case $P_{ext}=0$. However, the argument still applies when $P_{ext}\neq 0$ provided that $\partial_x P_{ext}(t,x)=0$ when 
$x=0$ or $x=L$.}  orthogonally:
\be\label{b9a}
\eta_x=0\quad\text{for}\quad x=0\text{ or }x=L.
\ee

Following Zakharov~\cite{Zakharov1968} and Craig--Sulem~\cite{CrSu}, we 
work with the evaluation of~$\phi$ at the free boundary
$$
\psi(t,x)\defn\phi(t,x,\eta(t,x)).
$$
Notice that $\phi$ is fully determined by its trace $\psi$ since $\phi$ is harmonic and satisfies $\phi_n=0$ on the walls and the bottom. Now, to obtain a system of two evolution equations for~$\eta$ and 
$\psi$, one introduces the Dirichlet to Neumann operator~$G(\eta)$ 
that relates~$\psi$ to the normal derivative of the potential by 
$$
G(\eta)\psi=\sqrt{1+\eta_x^2}\, \phi_n \arrowvert_{y=\eta}=(\phi_y-\eta_x \phi_x)\aeta.
$$
Then, it follows from \e{b8} that $\partial_t \eta=G(\eta)\psi$. Directly from \e{b5} we infer that
\begin{align}
&\partial_t \psi+g\eta +N(\eta)\psi+g\eta=-P_{ext},\quad\text{where}\notag\\
&
N(\eta)\psi=\mathcal{N} \big\arrowvert_{y=\eta}\quad\text{with}\quad
\mathcal{N}=
\mez\phi_x^2-\mez \phi_y^2+\eta_x \phi_x \phi_y.\label{b90}
\end{align}
With these notations, the water-wave system reads
\be\label{systemG}
\left\{
\begin{aligned}
&\partial_t \eta=G(\eta)\psi,\\
&\partial_t \psi+g\eta +N(\eta)\psi=-P_{ext}.
\end{aligned}
\right.
\ee

Introduce the energy $\mathcal{H}$, which is the 
sum of the potential and kinetic energies:
\be\label{b2}
\mathcal{H}(t)=\frac{g}{2}\int_0^L\eta^2(t,x)\dx
+\mez\iint_{\Omega(t)}\la \nabla_{x,y}\phi(t,x,y)\ra^2\dx\dy.
\ee
If $P_{ext}=0$, then $\Hr(t)=\Hr(0)$ for all time. 
Our goal is to find $P_{ext}$ such that:
\begin{enumerate}[(i)]
\item the variation of the external pressure are localized in the absorbing beach:
$$
{\text{ supp}}\, \px P_{ext}(t,\cdot)\subset [L-\delta,L]
$$
where $\delta>0$ is the length of the absorbing beach,
\item \label{ii} $\Hr$ is decreasing,
\item \label{iii} there exists a positive constant $C$ such that 
\be\label{b190}
\int_0^T\Hr(t)\dt\le C\Hr(0).
\ee
\end{enumerate}
One deduces from \e{ii} and \e{iii} that
$$
\Hr(T)\le \frac{1}{T} \int_0^T\Hr(t)\dt\le \frac{C}{T}\Hr(0),
$$
which implies an {\em exponential decay of the energy}. 
Indeed, for $T\ge 2C$, this gives $\Hr(T)\le 2^{-1} \Hr(0)$ and hence $\Hr(nT)\le 2^{-n}\Hr(0)$.

\subsection{Integral identity}
To prove the key estimate \e{b190}, the main difficulty is to compute the integral of the energy $\mathcal{H}$. 
To do so, we will prove an exact integral identity, of the form 
\be\label{b180}
\int_0^T (\Hr(t)+I(t))\dt =\int_0^T (W(t)+O(t)+N(t))\dt+B,
\ee
where the following properties hold:
\begin{itemize}
\item $I\ge 0$ and hence \e{b180} gives an upper bound for $\int_0^T\mathcal{H}\dt$.
\item $W$ depends on the pressure (if $P_{ext}=0$ then $W=0$).
\item $O$ is an {\em observation term} which means that it depends only on the behavior of the solutions near the wall $\{x=L\}$ (in the identity \e{b49} below 
this requires to chose $m=x$ for $x\in [0,L-\delta]$). 
\item $B$ is of the form $B=\int_0^L (F(T,x)-F(0,x))\dx$, for some function $F$. 
The key feature of this term is that, since it is not an integral in time, we can neglect $B$ for $T$ large enough.
\item $N$ is a cubic term while the energy $\mathcal{H}$ and the terms $I,W,O,B$ are quadratic terms. 
This implies that, for the linearized water-wave equations, the same identity holds with $N=0$. So the only difference between the nonlinear problem and the linear one is described by~$N$. 
Perhaps surprisingly, this term has a simple expression. Indeed, it is given by
$$
N(t)=\iint_{\Omega(t)} \rho_x \phi_x \phi_y\dydx-\int_0^L \frac{\rho}{2}\phi_x^2(t,x,-h)\dx,
$$
for some function $\rho$ depending linearly on $\eta$. A key point is that $N(t)\le \Hr(t)+I(t)$ for $\rho$ small enough.
\end{itemize}

In this paper we consider regular solutions of the water-wave system \e{systemG}. 
We postpone the definition of a regular solution to \S\ref{S:smooth} (see Definition~\ref{D:4.3}). 
Let us mention that, essentially, this definition is quite general since 
we only require that the free surface elevation $\eta(t,x)$ is $C^2$ in $x$ and the velocity $\nabla_{x,y}\phi\arrowvert_{y=\eta(t,x)}$ is $C^1$ in $x$. 

Here is our first main result.
\begin{theo}\label{T1bis}
Let $m\in C^\infty([0,L])$ be such that $m(0)=m(L)=0$ and set 
$$
\zeta=\partial_x(m\eta)-\uq \eta+\frac{1-m_x}{2}\eta,\quad 
\rho=(m-x)\eta_x+\left(\frac{5}{4}+\frac{m_x}{2}\right)\eta.
$$
Then, for any pressure $P_{ext}=P_{ext}(t,x)$ and any regular solution of \e{systemG} defined on the time interval $[0,T]$, 
there holds
\be\label{b49}
\begin{aligned}
\mez \int_0^T \Hr(t)\dt+\mathcal{Q}&=
-\int_0^T\int_0^L P_{ext}\, \zeta\, \dxdt-\int_0^L \zeta \psi \dx\ST\\
&\quad +\int_0^T\int_0^L \left(\frac{1-m_x}{2}\psi +(x-m)\psi_x\right)G(\eta)\psi\dxdt \\
&\quad +\int_0^T\iint_{\Omega(t)} \rho_x \phi_x \phi_y\dydxdt,
\end{aligned}
\ee
where
\be\label{b62}
\mathcal{Q}= \int_0^T\int_0^L \left(\frac{h}{2}+\frac{\rho}{2}\right)\phi_x^2(t,x,-h)\dxdt+\frac{L}{2}\int_0^T\int_{-h}^{\eta(t,L)}\phi_y^2(t,L,y)\dydt. 
\ee
\end{theo}

\begin{rema}
$(i)$ In \cite{Boundary} we proved a similar identity when $m(x)=x$ (assuming that $P_{ext}=0$). 
This weight does not vanish on $x=L$ and the identity proved in \cite{Boundary} 
was used to deduce only a boundary observability result. As explained in \S\ref{S:multiplier}, one cannot 
exploit easily this boundary observability result to study the stabilization problem. 
By contrast, the previous identity will allow us to study this problem.

$(ii)$ Assume that $P_{ext}=0$ and consider a small enough solution. Then, firstly, $\mathcal{Q}\ge 0$ and, secondly,  
one can absorb the term involving $\rho_x \phi_x\phi_y$ in the left-hand side. 
Since $\int_0^T \Hr(t)\dt=T\Hr(T)$ (since $P_{ext}=0$), we see that, loosely speaking, 
taking $T$ large enough, one can also absorb $\int_0^L \zeta \psi \dx\big\arrowvert_{0}^T$ in the left-hand side 
(as explained in \cite{Boundary}, to justify this argument requires some effort). Then we obtain an observability inequality, that is an 
estimate of the energy by means of the observation term $\int_0^T\int_0^L \left(\frac{1-m_x}{2}\psi +(x-m)\psi_x\right)G(\eta)\psi\dxdt$ 
(if $m(x)=x$ for $0\le x\le L-\delta$, the latter expression 
depends only on the behavior of $\eta,\psi$ in the neighborhood of $\{x=L\}$). 
In Appendix~\ref{A:Identity} we prove another integral identity which involves another observation term. 

\end{rema}

\subsection{Choice of the external pressure --- Hamiltonian damping}\label{S:HD}
As already mentioned, if $P_{ext}=0$ then the energy is conserved, that is $\Hr(t)=\Hr(0)$. 
Our goal is to find $P_{ext}$ so that the energy converges to zero. 

For the approach developed in this paper, 
there are five simple principles which govern the choice of $P_{ext}$: 
\begin{enumerate}
\item\label{Pr1} The energy $\Hr(t)$ must be {\em decreasing}. 

\item\label{Pr2} The {\em Cauchy problem for~\e{systemG} has to be well-posed}.

\item\label{Pr3} One could think that the stronger the damping, the faster the decay. However, 
in a somewhat counter-intuitive way, this is not the case. 
As will be clear in the proof, 
we need a bound of $P_{ext}$ in terms of the energy. 

\item\label{Pr4} {\em Localization}: we require that the derivative of the pressure $P_{ext}$ is localized in a neighborhood of $x=L$. 

\item\label{Pr5} {\em Boundary condition}: 
as already mentioned, to propagate the right-angle condition between 
the free surface and the wall (see \e{b9a}), 
the pressure $P_{ext}$ must satisfy $\px P_{ext}(t,x)=0$ for $x=L$.
\end{enumerate}

In this paragraph we give an expression for $P_{ext}$ in terms of the unknowns such that the above five conditions are satisfied. 

We begin by computing the work done by the pressure $P_{ext}$. In doing so, it is convenient to exploit the hamiltonian structure of the equation. 
Recall from Craig--Sulem (\cite{CrSu}) 
that $\Hr$ can be expressed as a function 
of $\eta$ and $\psi$,
$$
\mathcal{H}=\mez \int_0^L \big(g\eta^2+\psi G(\eta)\psi\big)\dx.
$$
Then, as observed by Zakharov~\cite{Zakharov1968}, the water-wave system can be written as
\footnote{The computations by Zakharov in \cite{Zakharov1968} are written only for periodic waves and $P_{ext}=0$ but 
the argument holds also in a rectangular tank with an external pressure. }
\be\label{b153}
\left\{
\begin{aligned}
&\frac{\partial\eta}{\partial t}=\frac{\delta \mathcal{H}}{\delta \psi},\\
&\frac{\partial\psi}{\partial t}=-\frac{\delta \mathcal{H}}{\delta \eta}-P_{ext}.
\end{aligned}
\right.
\ee
Then write
$$
\frac{d\mathcal{H}}{dt}=\int \left[\frac{\delta \mathcal{H}}{\delta \eta}\frac{\partial\eta}{\partial t}+\frac{\delta \mathcal{H}}{\delta \psi}\frac{\partial\psi}{\partial t}\right] \dx=
\int \left[\frac{\delta \mathcal{H}}{\delta \eta}\frac{\delta \mathcal{H}}{\delta \psi}-\frac{\delta \mathcal{H}}{\delta \psi}\frac{\delta \mathcal{H}}{\delta \eta}-\frac{\delta \mathcal{H}}{\delta \psi}P_{ext}\right] \dx,
$$
to deduce
\be\label{s1}
\frac{d\mathcal{H}}{dt}=-\int \frac{\delta \mathcal{H}}{\delta \psi}P_{ext} \dx=-\int_0^L \frac{\partial\eta}{\partial t} P_{ext}  \dx=-\int_0^L   P_{ext} G(\eta)\psi \dx.
\ee
This identity can be obtained directly from the definition \e{b2} of the energy, using the equations and the Stokes' formula. 

Since we want to force the energy to decrease to $0$, this suggests to chose $P_{ext}=P_{ext}(t,x)$ under the form 
$P_{ext}=\chi G(\eta)\psi$ where $\chi\ge 0$ is a compactly supported function satisfying $\chi=1$ on a neighborhood of $x=L$. 
As mentioned in the introduction, this choice is widespread and we pause to discuss it. 
Firstly, with this choice, the principles (P\ref{Pr1}) and (P\ref{Pr4}) are clearly satisfied. 
The principle (P\ref{Pr5}) is also satisfied since $G(\eta)\psi=\partial_t \eta$ and since $\partial_t\eta$ satisfies the same 
boundary condition \e{b9a} as $\eta$. To see that (P\ref{Pr3}) also holds, write
$$
\int_0^L P_{ext}(t,x)^2\, \dx =\int_0^L \big( \chi (\partial_t\eta)\big)^2\dx
\le (\sup \chi)\int_0^L\partial_t\eta P_{ext}\dx=-(\sup \chi)\frac{d\mathcal{H}}{dt},
$$
where we used \e{s1}. It follows that we have the estimate
$$
\int_{0}^T \int_0^L P_{ext}(t,x)^2\, \dxdt\le (\sup\chi)(\Hr(0)-\Hr(T))\le (\sup\chi)\Hr(0).
$$
However, we are not able to prove that the Cauchy problem for \e{systemG} is well-posed when $P_{ext}$ is given by 
$\chi G(\eta)\psi$ (except for the linearized equations). 

So we need to use another choice for $P_{ext}$. 
In this direction, we make the following elementary observation: by definition of $G(\eta)\psi$, it follows from the divergence theorem that
$$
\int_0^L   P_{ext} G(\eta)\psi \dx=\int_{\{y=\eta\}} P_{ext}\phi_n \dsigma=\iint \nabla_{x,y}\cdot (P_{ext}\nabla_{x,y}\phi)\dydx.
$$
Since $\Delta_{x,y}\phi=0$ and since $P_{ext}$ does not depend on $y$, it follows from \e{s1} that
$$
\frac{d\mathcal{H}}{dt}=-\int_0^L (\px P_{ext})\, \overline{V}\dx \quad\text{with}\quad
\overline{V}(t,x)=\int_{-h}^{\eta(t,x)}\phi_x(t,x,y)\dy.
$$
Since we want to force $\mathcal{H}$ to decrease, we set 
\be\label{p201a}
\px P_{ext}(t,x)=\chi(x) \overline{V}(t,x),
\ee
where $\chi\ge 0$ is a $C^\infty$ cut-off function satisfying $\chi=1$ on a neighborhood of $x=L$. The pressure $P_{ext}$ is defined up to a constant depending on time 
and to fix this constant we require that $P_{ext}(t,\cdot)$ has mean value $0$ on $(0,L)$. 

Clearly, with \e{p201a}, the conditions (P\ref{Pr1}), (P\ref{Pr4}) and (P\ref{Pr5}) are satisfied (for (P\ref{Pr5}) we use the boundary condition 
$\phi_x\arrowvert_{x=L}=0$ to obtain that $\px P_{ext}\arrowvert_{x=L}=0$). 
By contrast with the previous choice, we will see in \S\ref{S:4.2} 
that it is easy to prove that the Cauchy problem is well-posed, which means that the condition (P\ref{Pr2}) is now satisfied. 
Eventually, to see that (P\ref{Pr3}) also holds, we write
$$
\int_0^L (\px P_{ext})^2\, \dx =\int_0^L \big( \chi\overline{V}\big)^2\dx
\le (\sup \chi)\int_0^L(\partial_x P_{ext})\overline{V}\dx=-(\sup \chi)\frac{d\mathcal{H}}{dt},
$$
which shows that 
\be\label{p220}
\int_{0}^T \int_0^L (\px P_{ext}(t,x))^2\, \dxdt\le (\sup\chi)\Hr(0).
\ee

\subsection{A quantitative estimate}

Our second main result gives an inequality of the form 
$\Hr(T)\le (C/T)\Hr(0)$, for some constant $C$ depending on parameters which are considered fixed. 
As already mentioned, this will imply that, if the solution exists on time long time intervals of size $nT$ with $T\ge 2C$, then 
the energy converges exponentially fast to zero, so that 
$\Hr(T)\le 2^{-n}\Hr(0)$. In fact, we will obtain a weaker bound, of the form $\Hr(T)\le (C/\sqrt{T})\Hr(0)$.

A key feature of the water-wave problem is that the constant $C$ must depend on the frequency localization of $\eta$ and $\psi$. This can be easily understood by considering the linearized equations. Indeed, remembering that for these linear 
equations the dispersion relationship reads $\omega^2(k)=g|k|$, we see that 
high frequency waves propagate at a speed proportional 
to $|k|^{-1/2}$, which goes to $0$ when $|k|$ goes to $+\infty$. 
Now think of waves generated near $\{x=0\}$. The time needed to reach the absorption layer (located near $\{x=L\}$) will depend on the frequency, and 
moreover will goes to $+\infty$ when $|k|$ goes to $+\infty$. 
This explains that the result depends on the frequency localization of the solutions, 
in sharp contrast with the study of other wave equations. 
This observation goes back to Reid and Russell (\cite{ReidRussell1985}) who studied the controllability in infinite time of the linearized equations.

The following result gives a quantitative estimate of the form $T \Hr(T)\le C\Hr(0)$ where the constant $C$ depends on the frequency localization of the solutions. 
Since we consider the nonlinear equations, we cannot use Fourier analysis 
to measure the frequency localization of the solutions. We will consider 
instead some ratios 
between the energy and the $L^2$-norm of the derivatives of the unknown.

\begin{theo}\label{P:quant}
Denote by $\delta>0$ the length of the absorbing zone. 
Consider two functions $\chi,m$ in $C^\infty([0,L])$ such that:
\be\label{p364}
\begin{aligned}
&0\le \chi\le 1,\quad \chi(x)=1 \text{ if } x\in [L-\delta/2,L], \quad \chi(x)=0\text{ if }x\in  [0,L-\delta],\\[2ex]
&m(x)=x\text{ if }x\in [0,L-\delta/2],\quad m(L)=0.
\end{aligned}
\ee
Assume that
\be\label{b240}
\px P_{ext}(t,x)=\chi(x)\int_{-h}^{\eta(t,x)}\phi_x(t,x,y)\dy \quad\text{and}\quad \int_0^L P_{ext}(t,x)\dx =0,
\ee
and introduce the functions
\begin{align*}
\rho&=(m-x)\eta_x+\left(\frac{5}{4}+\frac{m_x}{2}\right)\eta,\\[1ex]
\Psi_1&=-m\psi_x-\uq \psi +\frac{1-m_x}{2} \psi,\\[1ex]
\Psi_2&= \px \left(\frac{1-m_x}{2}\psi +(x-m)\psi_x\right),
\end{align*}
and
set
\begin{align*}
C(m)&=\sup_{x\in [0,L]}  m(x)+\frac{L}{2}\sup_{x\in [0,L]} \la 1/2-m_x(x)\ra,\\
N_{1,T}(\psi)&=\sup_{t\in [0,T]}\frac{(\int \Psi_1(t,x)^2\dx)^{1/2}}{\Hr(0)^{1/2}},\\
N_{2,T}(\psi)&=\sup_{t\in [0,T]}\frac{(\int \Psi_2(t,x)^2\dx )^{1/2}}{\Hr(0)^{1/2}}.
\end{align*}
If $\rho$ satisfies
\be\label{p362}
\rho(t,x)\ge -h,\quad \la \rho_x(t,x)\ra\le c<\mez.
\ee
Then, for all $\alpha>0$ and for all regular solution of the water-wave system \e{systemG},
\be\label{p400}
T\left(\mez-c-\alpha\right)\Hr(T) \le \left(\frac{C(m)^2}{2\alpha g}+ \sqrt{T} N_{2,T}(\psi)
+\frac{2\sqrt{2}}{\sqrt{g}}N_{1,T}(\psi)\right)\Hr(0).
\ee
\end{theo}
\begin{rema}
Notice that $(\int \Psi_1(t,x)^2\dx)^{1/2}$ (resp.\ $(\int \Psi_2(t,x)^2\dx)^{1/2}$) 
is bounded by the $L^\infty_t(H^1_x)$-norm (resp.\ $L^\infty_t(H^2_x)$-norm) of $\psi$. 
One may wonder if these norms can be controlled on large time intervals, 
so that the previous estimate implies that $\Hr(T)\le c\Hr(0)$ with $c<1$. 
In \cite{Boundary}, assuming that $P_{ext}=0$, we prove such bounds 
for small enough initial data. The same result holds when 
$P_{ext}\neq 0$ is as in \e{b240} (the proof will be given in a separate paper where we will study the Cauchy problem). 
For the sake of completeness, we prove in the appendix such 
Sobolev estimates, uniformly in time, for the linearized equations (see Proposition~\ref{P:Uniform} and Remark~\ref{rema:Uniform}). 
So one may apply the previous result to these linear or weakly nonlinear settings.
For the nonlinear problem, in general, 
one cannot propagate Sobolev estimates on large time intervals (blow-up can occur, see \cite{CCFGGS}). 
However, the previous estimates seem reasonable for the typical low or medium frequency waves 
generated in a wave tank.
\end{rema}


\begin{coro}\label{Coro:2.4}
Consider two functions $\chi,m$ satisfying \e{p364} and a regular solution of \e{systemG} satisfying \e{p362}, as in the previous statement. 
Consider an integer $N$ and a real number $\beta>2$. 
Assume that the solution exists on a time interval $[0,T]$ with $T=N^\beta$ and is such that, on that time interval, we have the estimates
$$
N_{1,T}(\psi)\le N \text{ and } N_{2,T}(\psi)\le N.
$$
Then $\Hr(T)\le \exp\left( -\delta N^{-2}T\right)$ 
for some constant $\delta$ depending only on $m$.
\end{coro}
\begin{proof} 
Let $\alpha$ be such that $c+\alpha <1/2$. Then, for any $1\le T'\le T$, we have the bound
$$
\Hr(T')\le \frac{K N}{\sqrt{T'}}\Hr(0),
$$
for some constant $K$ depending only on $g$ and $m$. Since the problem is time-invariant, we see that the same estimate holds when $\Hr(T')$ is replaced by $\Hr((k+1)T')$ and $\Hr(0)$ by $\Hr(kT')$, provided that 
$(k+1)T'\le T$. We obtain that 
$$
\Hr(nT')\le \left(\frac{KN}{\sqrt{T'}}\right)^n \Hr(0).
$$
We conclude the proof by applying this inequality with $(n,T')$ such that $n$ is an integer, 
$T=nT'$ and 
$T'\ge (2K N)^2$.
\end{proof}

\section{Study of the Cauchy problem}\label{S:smooth}

We study here the Cauchy problem. 
In the first paragraph we consider the case $P_{ext}=0$. Our goal is to briefly recall 
from Alazard-Burq-Zuily \cite{ABZ4} how to solve 
the Cauchy problem for the water-wave equations in a rectangular tank. 
In the second paragraph we explain how to extend this result to the case $P_{ext}\neq 0$. 

\subsection{The homogeneous problem}

We recalled in the introduction that, for smooth enough solutions, 
the free surface must intersect the vertical walls of the tank orthogonally (see Section~$6$ in \cite{ABZ4}). 
This means that $\eta_x=0$ for $x=0$ or $x=L$. 
Now observe that $\psi_x=(\phi_x)\arrowvert_{y=\eta} +(\phi_y)\arrowvert_{y=\eta}\eta_x$. 
Since 
$\phi_x(t,x,y)=0$ for $x=0$ or $x=L$, we conclude that 
$\psi_x=0$ for $x=0$ or $x=L$. As a consequence, both $\eta$ and $\psi$ 
will belong to 
the following spaces.

\begin{defi}
Given a real number $\sigma>3/2$, one denotes by $H^\sigma_e(0,L)$ the space
$$
H^\sigma_e(0,L)=\{v\in H^\sigma(0,L): v_x=0\text{ for }x=0 \text{ or }x=L\},
$$
where $H^\sigma(0,L)$ denotes the usual Sobolev space of order $\sigma$.
\end{defi}

We first need to study the problem
\be\label{p200}
\begin{aligned}
&\Delta_{x,y}\phi=0 \quad &&\text{in}\quad && \Omega=\left\{\, (x,y)\,:\, x\in (0,L),~-h< y< \eta(x)\,\right\},\\
&\phi=\psi \quad&&\text{for}\quad&& y=\eta(x),\\
&\phi_x =0 \quad&&\text{for}\quad && x=0 \text{ or }x=L,\\
&\phi_y =0 \quad&&\text{for}\quad && y=-h.
\end{aligned}
\ee
The following regularity result is important since it implies that 
all the computations made in the proof 
are meaningful (these computations are either integrations 
by parts or consequences of the 
Green's identity). 

\begin{prop}[from \cite{Boundary}]\label{P6B}
If $(\eta,\psi)\in H^{\sigma}_{\pair}(0,L)\times H^\sigma_{\pair}(0,L)$ with 
$\sigma>5/2$, then there exists a unique variational solution to \e{p200} which satisfies 
$\nabla_{x,y}\phi\in C^1(\overline{\Omega})$.
\end{prop}
Since $\nabla_{x,y}\phi$ is continuous on $\overline{\Omega}$, one can define the Dirichlet to Neumann operator $G(\eta)$ 
by
$$
G(\eta)\psi=(\phi_y)\arrowvert_{y=\eta} -\eta_x(\phi_x)\arrowvert_{y=\eta}.
$$
Since $\nabla_{x,y}\phi\in C^1(\overline{\Omega})$, it follows that $G(\eta)\psi\in C^1([0,L])$. 
In fact, one can prove the following stronger regularity result: 
If $(\eta,\psi)\in H^{\sigma}_{\pair}(0,L)\times H^\sigma_{\pair}(0,L)$ with 
$5/2<\sigma<7/2$, then the traces $\phi_x\arrowvert_{y=\eta}$ and 
$\phi_y\arrowvert_{y=\eta}$ belong to $H^{\sigma-1}_e(0,L)$. 
Since $\eta_x$ also belongs to $H^{\sigma-1}_e(0,L)$, it follows from the usual product rule in Sobolev spaces that
$$
G(\eta)\psi\in H^{\sigma-1}_e(0,L).
$$
Similarly, the nonlinear expression $N(\eta)\psi$ defined by
\be\label{b90bis}
N(\eta)\psi=\mathcal{N} \big\arrowvert_{y=\eta}\quad\text{with}\quad
\mathcal{N}=
\mez\phi_x^2-\mez\phi_y^2+\eta_x \phi_x \phi_y,
\ee
is well-defined and satisfies $N(\eta)\psi\in H^{\sigma-1}_e(0,L)$. 

We now consider the Cauchy problem for the water-wave equations with $P_{ext}=0$, 
\be\label{systemGbis}
\left\{
\begin{aligned}
&\partial_t \eta=G(\eta)\psi,\\
&\partial_t \psi+g\eta +N(\eta)\psi=0,\\
&(\eta,\psi)\arrowvert_{t=0}=(\eta_0,\psi_0).
\end{aligned}
\right.
\ee

\begin{defi}\label{D:4.3a}
We say that $(\eta,\psi)$ is a regular solution of \e{systemGbis} 
provided that, for some $\sigma >5/2$, one has
$$
(\eta,\psi)\in C^{0}\big([0,T];H^{\sigma}_{e}(0,L)\times H^{\sigma}_e(0,L))\cap 
C^{1}\big([0,T];H^{\sigma-1}_{e}(0,L)\times H^{\sigma-1}_e(0,L)).
$$
\end{defi}
\begin{rema*}
We require $\sigma>5/2$ to be in a position to use Proposition~\ref{P6B}. Indeed, to justify all the computations below, we need that 
the gradient $\nabla_{x,y}\phi$ is $C^1$ up to the boundary. 
\end{rema*}

The following result (proved in \cite{ABZ4}, see also \cite{Boundary}) asserts that the water-wave equations have regular solutions.

\begin{prop}[from \cite{ABZ4}]\label{CP-h}
Consider an initial data $\eta_0,\psi_0$ in $H^s_e(0,L)$ for some real number $s\in (3,7/2)$. 
There exists $T>0$ and a unique 
solution
$$
(\eta,\psi)\in C^{0}\big([0,T];H^{s-\mez}_{e}(0,L)\cap H^{s}_e(0,L))\cap C^{1}\big([0,T];H^{s-\tdm}_{e}(0,L)\cap H^{s-1}_e(0,L)),
$$
to 
the Cauchy problem \e{systemGbis}.
\end{prop}
\begin{rema}
One can overcome the apparent loss of $1/2$-derivative by working with different unknowns (see \cite{Boundary} for further comments). 
However, the above result, with a simple statement, will be enough for our purposes.
\end{rema}

Let us briefly recall the strategy of the proof of Proposition~\ref{CP-h}. 
Consider an initial data $\eta_0,\psi_0\colon (0,L) \rightarrow \xR$ in $H^s_e(0,L)$ with $s>3$. 
Following Boussinesq (see~\cite[page 37]{Boussinesq}), the proof consists in extending these initial data 
to periodic functions, for which one can solve the Cauchy problem. Then one deduces 
the existence of solutions to the water-wave system in a tank by considering 
the restrictions of these solutions. 

To obtain periodic functions we use in \cite{ABZ4} a classical reflection/periodization procedure 
(with respect to the normal variable to the boundary of the tank). 
Notice that, in general, the even extension of a regular function on 
$(0,L)$ to a function defined on $(-L,L)$ is 
merely Lipschitz continuous (for instance one obtains $|x|$ starting from $x\mapsto x$). 
Now, the main difficulty is that there is no result 
which allows to handle Lipschitz free surface. 
However, when the free surface intersects the walls with a right angle, the reflected domain enjoys additional smoothness (namely up to $C^3$), 
which is enough to solve the Cauchy problem (this raises many other questions and we refer to \cite{ABZ4} for more details).

\subsection{The inhomogeneous problem}\label{S:4.2}

We now consider the inhomogeneous problem and assume that $P_{ext}$ satisfies 
\be\label{p201}
\px P_{ext}(t,x)=\chi(x) \overline{V}(t,x)\quad\text{with}\quad \overline{V}(t,x)=\int_{-h}^{\eta(t,x)}\phi_x(t,x,y)\dy,
\ee
and where $\chi$ is a $C^\infty$ cut-off function. The pressure $P_{ext}$ is defined up to a time-dependent function 
and to fix $P_{ext}$ we require that $P_{ext}$ has mean value $0$ on $(0,L)$.

\begin{defi}\label{D:4.3}
As above, we say that $(\eta,\psi)$ is a regular solution to the water-wave equations (see \e{systemGter}) 
provided that, for some $\sigma >5/2$, one has
$$
(\eta,\psi)\in C^{0}\big([0,T];H^{\sigma}_{e}(0,L)\times H^{\sigma}_e(0,L))\cap C^{1}\big([0,T];H^{\sigma-1}_{e}(0,L)\times H^{\sigma-1}_e(0,L)).
$$
\end{defi}
Hereafter, we assume that the initial data satisfies the so-called Taylor sign condition. 
The Taylor sign condition states that the pressure increases going from the air into the fluid domain. 
It is always satisfied when there is no pressure  
(see \cite{WuJAMS,LannesLivre}).

\begin{prop}\label{CP-i}
Consider an initial data $\eta_0,\psi_0$ in $H^s_e(0,L)$ for some real number $s\in (3,7/2)$, satisfying the Taylor sign condition. 
There exist $T>0$ and a unique 
solution
$$
(\eta,\psi)\in C^{0}\big([0,T];H^{s-\mez}_{e}(0,L)\times H^{s}_e(0,L)\big)\cap C^{1}\big([0,T];H^{s-\tdm}_{e}(0,L)\times H^{s-1}_e(0,L)\big),
$$
to 
the Cauchy problem
\be\label{systemGter}
\left\{
\begin{aligned}
&\partial_t \eta=G(\eta)\psi,\\
&\partial_t \psi+g\eta +N(\eta)\psi=P_{ext},\\
&(\eta,\psi)\arrowvert_{t=0}=(\eta_0,\psi_0),
\end{aligned}
\right.
\ee
where $P_{ext}$ is given by \e{p201}.
\end{prop}

We claim that this result follows from the proof of Proposition~\ref{CP-h}. To see this, we have to check two different properties. 

The first remark to be made is that the previous reflection/periodization procedure applies with a source term $P_{ext}$ provided 
that $P_{ext}$ has the same parity as $\psi$. Here, since after reflection, $\eta$ and $\phi$ are even in $x$, 
the function $\overline{V}$ is odd in $x$ and hence $P_{ext}$ is even in $x$. Since $\psi$ is also even in $x$, we verify that 
$P_{ext}$ and $\psi$ have the same parity. 

Secondly, we need to know the effect of $P_{ext}$ on the Sobolev energy estimates used in the analysis of the Cauchy problem. 
The key point is that $P_{ext}$ is a lower order term which can be handled as a source term in all energy estimates. 
This is where we use in a crucial way the choice of the pressure term. 
Indeed, we claim that
\be\label{p202}
\px P_{ext}(t,x)=-\chi(x) \int_0^x G(\eta)\psi (t,X) \dX.
\ee
To see this, recall that $\Delta_{x,y}\phi=0$ and $\phi_n=0$ for $x=0$ and $y=-h$. 
With $Q(x)=\{(X,y)\,\colon\, X\in [0,x],~-h\le y\le \eta(X)\}$, 
the divergence theorem 
implies that
$$
0=\int_{\partial Q(x)} \phi_n \dsigma=\int_{-h}^{\eta(t,x)}\phi_x(t,x,y)\dy+\int_{\substack{ y=\eta(X)\\X\in [0,x]}}\phi_n \dsigma. 
$$
This yields the well-known formula (see \S3.5 in \cite{LannesLivre})
\be\label{p410}
\overline{V}(t,x)+\int_0^x G(\eta)\psi (t,X) \dX=0,
\ee
which implies \e{p202}. 
Now, if $(\eta,\psi)\in H^{s-\mez}_{\pair}(0,L)\times H^{s-\mez}_{\pair}(0,L)$ with 
$3<s<7/2$, we have already recalled that $G(\eta)\psi$ belongs to $H^{s-3/2}_e(0,L)$. 
The previous formula implies that $P_{ext}$ belongs to $H^{s+1/2}_e(0,L)$. 
It turns out that this is exactly the regularity needed to consider $P_{ext}$ as a source term\footnote{For the sake of conciseness, we will not enter into the details. We mention the recent work by M{\'e}linand \cite{Melinand}Ê
where the author studies several questions about the water-wave problem with a source term. However, the well-posedness result in \cite{Melinand} applies for 
smoother initial data which is insufficient to prove Proposition~\ref{CP-i}. Nevertheless, an inspection of the analysis in \cite{ABZ4} 
shows that, for any $s>3$, one can consider a source term $P_{ext}$ provided that $P_{ext}\in L^1(0,T;H^{s+1/2}(0,L))$.}.

\section{Strategy of the proof: Introduction to the multiplier method}\label{S:multiplier}

The control theory of wave equations is well developed and many techniques have been introduced  
(microlocal analysis, Carleman estimates...). In this paper, we use the multiplier method. The key point is that this method 
allows us to work 
directly at the level of the nonlinear equations. 

For the sake of readability,  we begin by recalling some well-known results for the linear wave equation 
\be\label{w}
\partial_t^2u-\Delta u=0\quad \text{in }\Omega\subset \xR^n, \quad u\arrowvert_{\partial\Omega}=0.
\ee
The multiplier method, introduced by Morawetz, consists in multiplying the equations by 
$m(x)\cdot \nabla u(t,x)$, for some well-chosen function $m$, and to integrate by parts in space and time. 
For instance, by considering a smooth extension $m\colon \Omega\rightarrow \xR^n$ of the normal $\nu(x)$ to the boundary $\partial\Omega$, one obtains
\be\label{hr}
\int_0^T \int_{\partial\Omega}( \partial_n u)^2\dsigma\dt \le K(T)\mathcal{E}(u) \text{ where } \mathcal{E}(u)\defn \lA u(0,\cdot)\rA_{H^1_0(\Omega)}^2+\lA \partial_t u(0,\cdot)\rA_{L^2(\Omega)}^2.
\ee
This is the so-called {\em hidden regularity} property. 
The name comes from the fact that, using energy estimates, 
one controls only the $C^0([0,T];L^2(\Omega))$-norm of $\nabla_x u$ by means of the right-hand side of \e{hr}, which is insufficient to control the left-hand side 
of \e{hr} by means of classical trace theorems.

Another key estimate is the so-called {\em boundary 
observability inequality}, which is, compared to \e{hr}, a reverse inequality where one can bound the norms of the initial data by the integral of $\partial_nu$ restricted to a 
domain $\Gamma_0\subset \partial\Omega$. Such an inequality 
can be obtained by the multiplier method applied in this way: fix $x_0\in \xR^n$ and set 
$$
\Gamma(x_0)=\big\{ x\in\partial\Omega~,~(x-x_0)\cdot \nu(x)>0\big\},\quad 
T(x_0)=2\max_{x\in \overline{\Omega}}\la x-x_0\ra.
$$
Then, multiplying the equation by $(x-x_0)\cdot\nabla u$ and integrating by parts, we get that, for $T>T(x_0)$,
\be\label{bo}
(T-T(x_0))\mathcal{E}(u)\le \frac{T(x_0)}{2}\int_0^T \int_{\Gamma(x_0)}( \partial_n u)^2\dsigma\dt.
\ee 
For more details about the previous two inequalities, we refer the reader to the SIAM Review article by Lions~\cite{Lions1988} and 
the books by Komornik~\cite{Komornik}, Micu and Zuazua \cite{MicuZuazua}, Tucsnak and Weiss~\cite{TW2009} and the lecture notes by Alabau-Boussouira in~\cite{CIME-control}. 

Now consider a domain $\omega$ surrounding 
$\Gamma(x_0)$. The proof of the hidden regularity property \e{hr} allows us to bound the right-hand side in \e{bo} 
by the sum of $C_1\mathcal{E}$ (where $C_1$ is {independent} of time) and the integral of $\la\nabla u\ra^2$ on $(0,T)\times \omega$. 
Then, for $T$ large enough, one can absorb the term $C_1\mathcal{E}$ in the left-hand side of 
\e{bo} to deduce the following {\em internal observability inequality}:
\be\label{io}
\mathcal{E}(u)\le C(T)
\int_0^T \int_{\omega}\la\nabla u\ra^2\dxdt.
\ee 
This inequality can be used to obtain directly a stabilization result for the following 
damped wave equation
$$
\partial_t^2 v-\Delta v+a(x)\partial_t v=0 \quad \text{in }\Omega\subset \xR^n, \quad v\arrowvert_{\partial\Omega}=0,
$$
where $a\in C^\infty_0(\Omega)$ is a non-negative function satisfying $a(x)=1$ for $x$ in $\omega\subset\subset \Omega$. One can write $v$ as $v=u+w$ where $u$ and $w$ are given by solving
\be\label{b156}
\begin{aligned}
&\partial_t^2u-\Delta u=0\quad \text{in }\Omega\subset \xR^n, \quad u\arrowvert_{\partial\Omega}=0,\\
&\partial_t^2w-\Delta w+a(x)\partial_tw=-a(x)\partial_t u\quad \text{in }\Omega\subset \xR^n, \quad w\arrowvert_{\partial\Omega}=0,\\
&u(0,\cdot)=v(0,\cdot),\quad \partial_t u(t,0)=\partial_t v(0,\cdot);\quad 
w(0,\cdot)=0,\quad \partial_t w(t,0)=0.
\end{aligned}
\ee
Using the internal observability inequality for $u$ and a straightforward estimate for $w$ based on the Duhamel formula, one can deduce that 
$\mathcal{E}(v)(t)\le ce^{-c't}$ for some positive constants $c,c'$.

Similar results are known for many other wave equations and we only mention the paper by 
Machtyngier~\cite{Machtyngier} (see also \cite{MachtyngierZuazua}) for the Schr\"odinger equation $i\partial_t u+\Delta u=0$. 
Biccari~\cite{Biccari} introduced recently the use of the multiplier method to analyze the interior controllability problem 
for the fractional Schr\"odinger equation $i\partial_tu+(-\Delta)^su=0$ with $s\ge 1/2$ in a $C^{1,1}$ bounded domain 
with Dirichlet boundary condition. The key difference between the Schr\"odinger equation ($s=1$) and the fractional equation (for $1/2\le s<1$) is that the 
latter is nonlocal.  This 
is a source of difficulty 
since one seeks an observability result involving integrals over small localized domains. 
In particular, a key technical difference is that one needs to compute $\int (-\Delta)^su(x\cdot\nabla u)\dx$. The result is called a {\em Pohozaev identity}, 
since Pohozaev introduce the use of the multiplier $x\cdot\nabla u$ to study properties of elliptic equations (we refer to \cite{Ros-Oton-Serra} for such identities for fractional Laplacians).

In our previous paper \cite{Boundary}, we introduce the use of the multiplier method to study the gravity water-wave equations. 
To compare with the study by Biccari, notice that the linearized gravity water-wave equations can be written as $i\partial_tu+(-\Delta)^su=0$ with $s=1/4$ 
and hence the assumption 
$s\ge 1/2$ does not hold. This is a key feature of the problem since the group velocity is $\la\xi\ra^{2s-1}$ and hence,  
for $s<1/2$, high frequency waves propagate at a speed which goes to $0$ when $|\xi|$ goes to $+\infty$. 
Also, in \cite{Biccari,Ros-Oton-Serra}, the authors consider the case where $\Delta$ is the Laplacian with 
Dirichlet boundary condition while we consider periodic functions here. 
More importantly, the main difficulty in \cite{Boundary} or in the present paper is that the equations are nonlinear. In particular, 
we need a Pohozaev identity for $\int (G(\eta)\psi)(x\cdot\nabla \psi)\dx$ where $G(\eta)$ is an operator with variable coefficients.

Let us now explain the main difficulties one has to cope with to stabilize the water-wave equations. 
Firstly, one cannot decouple the problem of the observability and the question of the stabilization. 
Compared to what is done for the wave equation (see \e{b156}), since the water-wave system 
is quasi-linear, one cannot write the solution as the sum of the two different problems. This means that one cannot assume 
that $P_{ext}=0$ for the purpose of proving observability. Another difficulty is that we do not know how to deduce an internal observability inequality from a boundary observability inequality (for the wave equation or the Schr\"odinger equation, as we recalled above, this is possible thanks to a hidden regularity result). To overcome these two problems, guided by the lectures notes by Alabau-Boussouira~(\cite{CIME-control}), 
we prove directly an internal observability result for the water-wave system by considering 
a multiplier $m(x)\partial_x$ with $m(x)=x\kappa(x)$ where $\kappa$ is a cut-off function satisfying $\kappa(x)=1$ for $0\le x\le L-\delta$ and $\kappa(x)=0$ for $L-\delta/2\le x\le L$.

\section{Proof of Theorem~\ref{T1bis}}\label{S:5}

The proof is in four steps. 

\begin{nota*}
We write simply
$$
\int \, dx ,\quad \int \, dy ,\quad  
\int\,dt,
$$
as shorthand notations for, respectively,
$$
\int_0^L \, dx ,\quad \int_{-h}^{\eta(t,x)} \, dy ,\quad 
\int_0^T \,dt.
$$
\end{nota*}

\bigbreak

\noindent\textbf{Step 1 : the multiplier method.} To estimate $\int \Hr(t)\dt$, we will use in a crucial way the unknown
$$
\Theta\defn - \eta\partial_t\psi-\frac{g}{2}\eta^2.
$$
In Appendix~\ref{Luke}, we will see that this function is related to Luke's variational principle. 
This observation explains that we will be able to compare $\iint \Theta \dxdt$ and $\int \Hr(t)\dt$. 
The function $\Theta$ was introduced in \cite{Boundary} for the purpose 
of proving a boundary observability result. In that reference, we used the weight $m(x)=x$. 
Now, for a general weight $m(x)$, to obtain an identity for $\iint \Theta \dxdt$ we proceed in a different way. We write
$$
\iint \Theta \dxdt=\iint m_x\Theta \dxdt+\iint (1-m_x)\Theta \dxdt.
$$
The second term in the right-hand side is an observation term. Indeed, 
if $m(x)=x\kappa(x)$ where 
$\kappa(x)=1$ in $[0,L-\delta]$ and $\kappa(x)=0$ in $[L-\delta/2,L]$, then $(1-m_x)\Theta$ depends only 
on the behavior of $\eta$ and $\psi$ in a neighborhood of $x=L$. 
So the key point is to obtain an identity for 
$\iint m_x\Theta \dxdt$. This is the purpose of the following lemma.

\begin{lemm}\label{P1}
Consider a smooth solution of the water-wave system and a smooth function $m\colon [0,L]\rightarrow \xR$ satisfying $m(0)=m(L)=0$. Then one has
$$
\iint m_x \Theta \,\dxdt+R_a=-\int \px(m \eta) \psi\, \dx\ST-\iint P_{ext}m\eta_x\, \dxdt,
$$
where
\be\label{b31b}
R_a=\iint (G(\eta)\psi)m\psi_x\dxdt+ \iint (N(\eta)\psi)m\eta_x\dxdt.
\ee
\end{lemm}
\begin{proof}
The proof is based on the multiplier method applied in the following way: 
instead of multiplying the equations by $(m\px\eta,m\px\psi)$, we set
$$
A\defn \iint \big[(\partial_t\eta)(m\px\psi)-(\partial_t\psi)(m\px\eta)\big]\,\dxdt,
$$
and we compute $A$ in two different ways. Then the wanted identity 
will be deduced by comparing the two results. 

{\em First computation.} 
Since $m(0)=m(L)=0$, directly from the definition of $A$, using integration by parts in space and time,  
one has
$$
A=\int m \eta \psi_x\dx\ST+\iint m_x \eta\partial_t\psi \,\dxdt.
$$
Since $m(0)=m(L)=0$ one can further integrate by parts in $x$ in the first term to obtain
\be\label{b32}
A=-\int \px(m \eta) \psi\dx\ST+\iint m_x \eta\partial_t\psi \,\dxdt.
\ee

{\em Second computation.} We simply compute $A$ by replacing $\partial_t\eta$ and $\partial_t\psi$ by the expressions given by System~\e{systemG}. We find that
\be\label{b34}
A=\iint \left(P_{ext}+g\eta\right)m\eta_x\dx+R_a
\ee
where $R_a$ is given by \e{b31b}. On the other hand, since $m(0)=m(L)=0$, integrating by parts, we obtain
$$
-\int g\eta m\eta_x\dx=\mez \int g m_x \eta^2\dx.
$$
By combining this identity with \e{b34}, it follows that
$$
A=-\mez \int g m_x \eta^2\dx+\iint P_{ext}m\eta_x\dx+R_a.
$$
Then, by comparing the previous identity with \e{b32} we conclude the proof.
\end{proof}

\bigbreak

\noindent\textbf{Step 2: equipartition of the energy.} Introduce the average in time kinetic (resp.\ potential) energy 
denoted by $A_K$ (resp.\ $A_P$). By definition,
$$
A_K=\mez \iint \psi G(\eta)\psi\dxdt,\quad 
A_P=\frac{g}{2}\iint \eta^2\dxdt,
$$
and we have
\be\label{b41}
\int \Hr(t)\dt =A_K+A_P.
\ee
The analysis below relies heavily on the idea of comparing $A_K$ and $A_P$. 
We will see that one has equipartition of the energy, which means that the difference between these two quantities can be handled as a remainder term. 
We will not only compare $A_K$ and $A_P$ but also some localized versions where we add an extra factor $\chi=\chi(x)$ in the integrals.

\begin{lemm}\label{L3}
For any smooth function $\chi=\chi(x)$, there holds
\be\label{b70rho}
\begin{aligned}
\frac{g}{2}\iint \chi \eta^2\dxdt
&=\mez \iint \chi \psi G(\eta)\psi\dxdt\\
&\quad-\mez \iint \chi \eta P_{ext}\dxdt 
-\mez \int\chi \eta\psi\dx\ST\\
&\quad -\mez\iint \chi \eta \, N(\eta)\psi\dxdt.
\end{aligned}
\ee
In particular, with $\chi=1$, one has
\be\label{b43}
A_K-A_P=\mez \iint \eta P_{ext}\dxdt +R_b
+\mez \int \eta\psi\dx\ST,
\ee
where
$$
R_b=\mez \iint \eta N(\eta)\psi \dxdt.
$$
\end{lemm}
\begin{proof}
Using $\partial_t\eta=G(\eta)\psi$ and integrating by parts, we find that
\begin{align*}
\mez \iint \chi \psi G(\eta)\psi\dxdt-\frac{g}{2}\iint \chi \eta^2\dxdt&=\mez\iint \chi\left[ \psi (\partial_t\eta) -g\eta^2\right]\dxdt\\
&=\mez \iint \chi\left[ -\eta (\partial_t \psi+g\eta)\right]\dxdt \\
&\quad+\mez \int \chi\eta \psi\dx \ST.
\end{align*}
So \e{b70rho} follows from the equations \e{systemG} for $\psi$.
\end{proof}

By combining the previous identities, we will deduce the following lemma.

\begin{lemm}\label{L4}
Set 
$$
\zeta=\partial_x(m\eta)-\uq \eta+\frac{1-m_x}{2}\eta
$$
There holds
\be\label{b47bis}
\begin{aligned}
\mez \int \Hr(t)\dt&=
-\iint P_{ext}\, \zeta\, \dxdt\\
&\quad +\mez \iint (1-m_x)\psi G(\eta)\psi\dxdt\\
&\quad-\int \zeta \psi \dx\ST\\
&\quad -\iint (G(\eta)\psi)m\psi_x\dxdt\\
&\quad -\iint \zeta \left(N(\eta)\psi\right)\dxdt.
\end{aligned}
\ee

\end{lemm}
\begin{proof}
Recall that ${\displaystyle \Theta=-\eta\partial_t\psi-\frac{g}{2}\eta^2}$. Then, using 
the equation \e{systemG} for $\psi$, we get
$$
\Theta=-\eta(\partial_t\psi+g\eta)+\frac{g}{2}\eta^2=\eta\left(P_{ext}+N(\eta)\psi\right)+\frac{g}{2}\eta^2,
$$
which implies that
\be\label{b45}
\iint m_x\Theta \dxdt =\frac g 2 \iint m_x \eta^2\dxdt +\iint P_{ext} m_x \eta \dxdt
+R_c
\ee
where $R_c$ is given by
\be\label{b168}
R_c=\iint m_x\eta\left(N(\eta)\psi\right)\dxdt.
\ee
Now recall from Lemma~\ref{P1} that
$$
\iint m_x \Theta \,\dxdt+R_a=\int m \eta \psi_x\, \dx\ST-\iint P_{ext}m\eta_x\, \dxdt.
$$
Then, it follows from \e{b45} that
$$
\frac g 2 \iint m_x \eta^2\dxdt 
+R_a+R_c=\int m \eta \psi_x\, \dx\ST-\iint P_{ext}\partial_x(m\eta)\, \dxdt.
$$
We then split the coefficient $m_x$ in the left-hand side as $m_x=1+(m_x-1)$ to obtain
\begin{align*}
\frac g 2 \iint  \eta^2\dxdt +R_a+R_c
&=\int m \eta \psi_x\, \dx\ST-\iint P_{ext}\partial_x(m\eta)\, \dxdt\\
&\quad+\frac{g}{2}\iint(1-m_x)\eta^2\dxdt.
\end{align*}
On the other hand, it follows from \e{b41} and \e{b43} that
\begin{align*} 
\frac g 2 \iint  \eta^2\dxdt &=A_P=\mez (A_K+A_P)+\mez (A_P-A_K)\\
&=\mez \int \Hr(t)\dt -\uq\iint \eta P_{ext}\dxdt -\mez R_b
-\uq \int \eta\psi\dx\ST.
\end{align*}
By combining the previous results, we get that
\be\label{b47aa}
\begin{aligned}
\mez \int \Hr(t)\dt&=
-\iint P_{ext}\left(\partial_x(m\eta)-\uq \eta\right)\, \dxdt\\
&\quad +\frac{g}{2}\iint(1-m_x)\eta^2\dxdt\\
&\quad -\iint (G(\eta)\psi)m\psi_x\dxdt\\
&\quad-\int \left(\px(m \eta)-\uq\eta\right) \psi\, \dx\ST\\
&\quad -\iint \left(\partial_x(m\eta)-\uq \eta\right) N(\eta)\psi \dxdt.
\end{aligned}
\ee

On the other hand, it follows from \e{b70rho} that
\be\label{b70}
\begin{aligned}
\frac{g}{2}\iint (1-m_x)\eta^2\dxdt
&=\mez \iint (1-m_x)\psi G(\eta)\psi\dxdt\\
&\quad-\mez \iint (1-m_x)\eta P_{ext}\dxdt 
-\mez \int(1-m_x)\eta\psi\dx\ST\\
&\quad -\mez\iint (1-m_x)\eta \, N(\eta)\psi\dxdt.
\end{aligned}
\ee
By plugging \e{b70} in \e{b47aa}, we obtain the desired identity \e{b47bis}.
\end{proof}

\bigbreak

\noindent\textbf{Step 3: a Pohozaev identity.} To complete the proof of the theorem, 
it remains to study the last two terms in the right-hand side of \e{b47bis}. 
We begin with the last but one term
$$
\int (G(\eta)\psi)m\psi_x\dx.
$$
To handle this term, we split it into two terms in order to obtain an expression which 
make appear a positive term through a Pohozaev identity.  So we write
\be\label{b65}
\int (G(\eta)\psi)m\psi_x\dx=\int (G(\eta)\psi)x\psi_x\dx+\int (G(\eta)\psi)(m-x)\psi_x\dx.
\ee
We now use the following Pohozaev identity proved in \cite{Boundary}. 
\begin{lemm}[from \cite{Boundary}]
One has
\be\label{b67}
\int (G(\eta)\psi)x\psi_x\dx=\Sigma+ \int (\eta-x\eta_x)\big(N(\eta)\psi\big)\dx,
\ee
where $\Sigma=\Sigma(t)$ is a positive term given by
$$
\Sigma(t)=\frac{h}{2} \int_0^L \phi_x^2(t,x,-h)\dx+\frac{L}{2}\int_{-h}^{\eta(t,L)}\phi_y^2(t,L,y)\dy. 
$$
\end{lemm}
Observe that, by integrating in time, we obtain \e{b62} with $\mathcal{Q}=\int_0^TQ(t)\dt$. 
By so doing, we end up with
\be\label{b47c}
\begin{aligned}
\mez \int \Hr(t)\dt+\int \Sigma(t)\dt &=
-\iint P_{ext}\, \zeta\, \dxdt\\
&\quad +\iint \left(\frac{1-m_x}{2}\psi +(x-m)\psi_x\right)G(\eta)\psi\dxdt \\
&\quad-\int \zeta \psi \dx\ST\\
&\quad -\iint \rho N(\eta)\psi \dxdt.
\end{aligned}
\ee
where the coefficient $\rho$ in the last term is given by
$$
\rho=\zeta+\eta-x\eta_x=(m-x)\eta_x+\left(\frac{5}{4}+\frac{m_x}{2}\right)\eta.
$$

\bigbreak

\noindent\textbf{Step 4: computation of the remainder term.} 
In view of a possible application to the stabilization problem, 
the previous identity \e{b47c} is not sufficient since one cannot control {\em a priori} the last term 
$\int \rho N(\eta)\psi \dx$ by means of the energy. Indeed, 
$$
N(\eta)\psi=\mathcal{N} \big\arrowvert_{y=\eta}\quad\text{with}\quad
\mathcal{N}=
\mez\phi_x^2-\mez\phi_y^2+\eta_x \phi_x \phi_y,
$$
and clearly one cannot simply use the previous definition to bound $N(\eta)\psi$ 
by $K\iint \la\nabla_{x,y}\phi\ra^2\dydx$ using only the trace theorem. However, 
as in \cite{Boundary}, inspired by the analysis done by Benjamin and Olver (\cite{BO}) of the conservation laws for water waves, 
one can rewrite $\int \rho N(\eta)\psi \dx$ as the sum of two terms which can be controlled either by the energy or by 
the positive term $\Sigma$ given by the Pohozaev identity.

\begin{lemm}
There holds
\be\label{b110}
\int \rho N(\eta)\psi\dx
=-\iint \rho_x\phi_x\phi_y \dydx +\mez \int \rho \phi_x^2\arrowvert_{y=-h}\dx.
\ee
\end{lemm}
\begin{proof}
This result will be obtained by writing $\int \rho N(\eta)\psi\dx$ under the form 
$$
\iint u(t,x,\eta(t,x))\dxdt+\iint f(t,x,\eta(t,x))\eta_x(t,x)\dxdt,
$$
together with an application of 
the following elementary identity: for any functions $u=u(x,y)$ and $f=f(x,y)$ with $f\arrowvert_{x=0}=f\arrowvert_{x=L}=0$, one has
\be\label{b41ab}
\int u(x,\eta(x))\dx+\int f(x,\eta)\eta_x\dx
\\=\iint (u_y-f_x) \dy\dx
+\int u(x,-h)\dx.
\ee
Indeed, 
$$
\int_0^L u(x,\eta(x))\dx=\int_0^L \int_{-h}^{\eta(x)}u_y(x,y)\dy 
+\int_0^L u(x,-h)\dx,
$$
and
$$
\int_0^L f(x,\eta)\eta_x\dx+\int_0^L \int_{-h}^{\eta(x)}f_x\, \dy\dx=\int_{-h}^{\eta} f\, \dy\dx\Big\arrowvert_{x=0}^{x=L}=0.
$$

Now, by definition of $N(\eta)\psi$, we have
$$
\int \rho N(\eta)\psi\dx=
\int u(x,\eta)\dx +\int f(x,\eta)\eta_x\dx,
$$
with
$$
u(x,y)=\mez \rho\left( \phi_x^2-\phi_y^2\right),\qquad
f(x,y)=\rho \phi_x\phi_y.
$$
Since $f\arrowvert_{x=0}=f\arrowvert_{x=L}=0$ and
$$
u\arrowvert_{y=-h}=\mez \rho \phi_x^2\arrowvert_{y=-h} ,\qquad 
u_y-f_x=-\rho_x\phi_x\phi_y,
$$
the desired result \e{b110} follows from \e{b41ab}.
\end{proof}
By plugging this result into \e{b47c}, we complete the proof of Theorem~\ref{T1bis}.

\section{Proof of Proposition~\ref{P:quant}}\label{S:6}

We want to prove an inequality of the form 
\be\label{p230}
T \Hr(T)\le C\Hr(0),
\ee
where $C$ is as given by the right-hand side of \e{p400}. 
To do so, we will prove that
\be\label{p231}
\int_0^T\Hr(t)\dt\le C\Hr(0).
\ee
Then the desired bound \e{p230} will be deduced from \e{p231}Ê
and the fact that the energy is decreasing, so that 
$T\Hr(T)\le \int_0^T\Hr(t)\dt$.

\begin{lemm}
Assume that $\rho$ satisfies
$$
\rho(t,x)\ge -h,\quad \la \rho_x(t,x)\ra\le c<\mez.
$$
Then
\be\label{b49c}
\begin{aligned}
\left(\mez-c\right)\intT \Hr(t)\dt&\le 
-\intTL P_{ext}\, \zeta\, \dxdt-\intL \zeta \psi \dx\ST\\
&\quad +\intTL \left(\frac{1-m_x}{2}\psi +(x-m)\psi_x\right)G(\eta)\psi\dxdt.
\end{aligned}
\ee
\end{lemm}
\begin{proof}
The assumptions on 
$\rho$ imply that $\mathcal{Q}\ge 0$ as well as the estimate
$$
\int_0^T\iint_{\Omega(t)} \rho_x \phi_x \phi_y\dydxdt
\le \frac{c}{2}\int_0^T\iint_{\Omega(t)} \left( \phi_x^2+\phi_y^2\right)\dydxdt\le c\int_0^T\Hr(t)\dt.
$$
The wanted inequality then immediately follows from Theorem~\ref{T1bis}.
\end{proof}

\begin{nota*}
We use the notations 
$$
\lA f\rA_{L^2}=\left(\int_0^L f(x)^2\dx\right)^{1/2},\quad 
\lA f\rA_{L^\infty}=\sup_{x\in [0,L]}\la f(x)\ra.
$$
\end{nota*}
\begin{lemm}
For any $\alpha>0$,
\begin{align*}
&-\intTL P_{ext}\, \zeta\, \dxdt
+\intTL \left(\frac{1-m_x}{2}\psi+(x-m)\psi_x\right)G(\eta)\psi\dxdt\\
&\qquad\le \frac{C(m)^2}{2\alpha g} \intT \lA \px P_{ext}\rA_{L^2}^2dt +\alpha\int_0^T \Hr(t)\dt\\
&\qquad\quad+N_2(\psi)\left(T \Hr(0) \int_0^T\lA \px P_{ext}\rA_{L^2}^2\dt\right)^{1/2},
\end{align*}
where
\begin{align*}
C(m)&=\lA m\rA_{L^\infty}+\frac{L}{2}\lA (1-m_x)-1/2\rA_{L^\infty},\\
N_2(\psi)&=\sup_{t\in [0,T]}\frac{\lA  \Psi_2(t)\rA_{L^2}}{\sqrt{\Hr(0)}}\quad\text{with }
\Psi_2=\px \left(\frac{1-m_x}{2}\psi +(x-m)\psi_x\right).
\end{align*}
\end{lemm}
\begin{proof}
We split $\intL P_{ext}\, \zeta\, \dx$ as the sum $A+B$ where
$$
A=\intL P_{ext}\px (m\eta)\dx,\qquad
B=\intL P_{ext}\left(-\uq \eta+\frac{1-m_x}{2}\eta\right)\dx.
$$
Since $m(0)=m(L)=0$, one has
$A=-\intL (\px P_{ext})m\eta\dx$, and hence
$$
\la A\ra\le \lA \px P_{ext}\rA_{L^2}\lA m\rA_{L^\infty}\lA \eta\rA_{L^2}.
$$
On the other hand,
$$
\la B \ra\le \mez \lA (1-m_x)-1/2\rA_{L^\infty}\lA P_{ext}\rA_{L^2}\lA \eta\rA_{L^2}.
$$
Now, since $P_{ext}$ has mean value zero by assumption \e{b240}, it follows from the Poincar\'e inequality that
\be\label{Poincare}
\lA P_{ext}\rA_{L^2}\le L \lA \px P_{ext}\rA_{L^2}.
\ee
By combining the previous inequalities, we conclude that
$$
\la \intL P_{ext}\, \zeta\, \dx\ra\le C(m) \lA \px P_{ext}\rA_{L^2}\lA \eta\rA_{L^2},
$$
which immediately implies that, for any $\alpha>0$,
\begin{align*}
\la \intTL P_{ext}\, \zeta\, \dxdt\ra &\le \frac{1}{2\alpha g} C(m)^2 \int_0^T \lA \px P_{ext}\rA_{L^2}^2\dt 
+\frac{\alpha g}{2}\int_0^T\lA \eta\rA_{L^2}^2\dt\\
&\le \frac{1}{2\alpha g} C(m)^2 \int_0^T \lA \px P_{ext}\rA_{L^2}^2\dt 
+\alpha\int_0^T\Hr(t)\dt.
\end{align*}
It remains to estimate the terms which involve the Dirichlet to Neumann operator. 
In doing so, we use the following well-known formula (which follows from \e{p410})
$$
G(\eta)\psi=-\px \overline{V}.
$$
Since $\overline{V}$ vanishes for $x=0$ or $x=L$, by integration by parts, we get
$$
\intL \left(\frac{1-m_x}{2}\psi+(x-m)\psi_x\right)G(\eta)\psi\dx=\intL \Psi_2 \overline{V}\dx.
$$
Since $\px P_{ext}=\chi \overline{V}$ by definition and since $\chi(x)=1$ on the support of $\Psi_2$ 
(by assumption on $m$), we deduce that
$$
\intL \left(\frac{1-m_x}{2}\psi+(x-m)\psi_x\right)G(\eta)\psi\dx =
\intL \Psi_2 \px P_{ext}\dx. 
$$
As a consequence,
\begin{align*}
\la \intL \left(\frac{1-m_x}{2}\psi +(x-m)\psi_x\right)G(\eta)\psi\dx\ra 
&\le \lA \Psi_2\rA_{L^2}\lA \px P_{ext}\rA_{L^2}\\
&\le N_{2,T}(\psi)\sqrt{\Hr(0)}\lA \px P_{ext}\rA_{L^2},
\end{align*}
by definition of $N_{2,T}(\psi)$. 
Then, using $\int_0^T f(t)\dt\le \sqrt{T}(\int_0^T f(t)^2\dt)^{1/2}$, we deduce that
\begin{multline*}
\la \intTL \left(\frac{1-m_x}{2}\psi +(x-m)\psi_x\right)G(\eta)\psi\dxdt\ra \\
\le N_{2,T}(\psi)\left(T \Hr(0) \int_0^T\lA \px P_{ext}\rA_{L^2}^2\dt\right)^{1/2}.
\end{multline*}
This completes the proof.
\end{proof}

In view of the previous lemmas, it remains only to estimate the integrals
$$
\int_{0}^T \int_0^L (\px P_{ext}(t,x))^2\, \dxdt,\quad \int \zeta \psi \dx\ST.
$$
Firstly, recall from \e{p220} that
\be\label{p220bis}
\int_{0}^T \int_0^L (\px P_{ext}(t,x))^2\, \dxdt\le \Hr(0),
\ee
here we used the assumption $\chi\le 1$.
To estimate the second term, set
$$
B(t)\defn\intL \zeta(t,x) \psi(t,x) \dx.
$$
So we have to estimate $B(T)-B(0)$. In fact, we will estimate the two terms separately. 
We begin by integrating by parts to write $B(t)$ under the form
$$
B(t)=\intL \eta \Psi_1\dx \quad\text{where}\quad\Psi_1\defn -m\psi_x-\uq \psi +\frac{1-m_x}{2} \psi.
$$
As a result $B(t)\le \lA \eta\rA_{L^2}\lA \Psi_1\rA_{L^2}$ and hence
$$
\la \intL \zeta \psi \dx\ST\ra\le \lA \eta(0)\rA_{L^2}\lA \Psi_1(0)\rA_{L^2}+\lA \eta(T)\rA_{L^2}\lA \Psi_1(T)\rA_{L^2}.
$$
Remembering that 
$$
N_{1,T}(\psi)=\sup_{t\in [0,T]}\frac{\lA  \Psi_1(t)\rA_{L^2}}{\sqrt{\Hr(0)}},\quad 
\lA \eta(t)\rA_{L^2}\le \sqrt{\frac{2}{g}}\sqrt{\Hr(t)},
$$
and using again the fact that $\Hr$ is decreasing, we obtain the estimate
$$
\lA \eta(0)\rA_{L^2}\lA \Psi_1(0)\rA_{L^2}+\lA \eta(T)\rA_{L^2}\lA \Psi_1(T)\rA_{L^2}
\le \frac{2\sqrt{2}}{\sqrt{g}}N_{1,T}(\psi)\Hr(0).
$$
By combining the previous estimates, we end up with
$$
\left(\mez-c-\alpha\right)\int_0^T\Hr(t)\dt \le \left\{\frac{C(m)^2}{2\alpha g}+ \sqrt{T} N_{2,T}(\psi)+\frac{2\sqrt{2}}{\sqrt{g}}N_{1,T}(\psi)\right\}\Hr(0).
$$
As explained at the beginning of this section, this completes the proof.

\appendix

\section{Uniform estimates for the linearized problem}\label{S:linear}

In this appendix we consider Cauchy problem for 
the linearized water-wave equations. As already seen, one can reduce the analysis of the Cauchy problem to the case of periodic functions which are even in $x$. 
We thus assume in this section that $x$ belongs to the circle $S^1=\xR/(2\pi\xZ)$ and use Fourier analysis. Also, to simplify notations we assume that $g=1$ and that 
the fluid is infinitely deep (that is $h=+\infty$), 
so that $G(0)$ is the Fourier multiplier $\la D_x\ra$ defined by $\la D_x\ra (\sum \psi_n e^{inx})
=\sum \psi_n \la n\ra e^{inx}$. The equations read
\be\label{b300}
\left\{
\begin{aligned}
&\partial_t \eta=\la D_x\ra\psi,\\
&\partial_t \psi+\eta +P_{ext}=0.
\end{aligned}
\right.
\ee
Set 
\be\label{b301}
P_{ext}=-\px^{-1}\left(\chi(x)\partial_x^{-1} \la D_x\ra\psi\right),
\ee
where $\chi\ge 0$ is a smooth compactly supported function, even in $x$, and where, by definition, 
$$
\px^{-1}\sum_{n\in \xZ} \psi_n e^{i nx}=\sum_{n\in \xZ\setminus\{0\}}  \frac{\psi_n}{in}e^{i nx}.
$$
For the linearized problem, this definition of $P_{ext}$ is equivalent to \e{p201a} (recall that we assume 
that $P_{ext}$ has mean value zero).
\begin{prop}[Uniform estimates]\label{P:Uniform}
Let $s\in [0,+\infty)$ be such that $2s\in \xN$. 
For any initial data $(\eta_0,\psi_0)$ in the Sobolev space 
$H^s(S^1)\times H^{s+\mez}(S^1)$, the Cauchy problem for \e{b300}-\e{b301} has a unique solution 
$(\eta,\psi)\in C^0([0,+\infty);H^s(S^1)\times H^{s+\mez}(S^1))$. Moreover, 
there exists a constant $C_s$ depending only on $s$ such that, for any $t\ge 0$,
\be\label{b346}
\lA \eta(t)\rA_{H^s}+\lA \psi(t)\rA_{H^{s+\mez}}\le C_s \lA \eta(0)\rA_{H^s}+C_s\lA \psi(0)\rA_{H^{s+\mez}}.
\ee
\end{prop}
\begin{rema}\label{rema:Uniform}
The quantities $N_{1,T}(\psi)$ and $N_{2,T}(\psi)$, as introduced in the statement of Theorem~\ref{P:quant}, 
are bounded by
$$
K_1(m)\frac{\lA \psi(t)\rA_{H^{1}}}{\mathcal{H}(0)},\quad 
K_2(m)\frac{\lA \psi(t)\rA_{H^{2}}}{\mathcal{H}(0)}.
$$
The previous proposition implies that 
$$
N_{1,T}(\psi)\les \frac{\lA (\eta(0),\psi(0))\rA_{H^\mez \times H^1}}
{\lA (\eta(0),\psi(0))\rA_{L^2\times \dot H^\mez}},\quad 
N_{2,T}(\psi)\les \frac{\lA (\eta(0),\psi(0))\rA_{H^{\tdm}\times H^2}}{\lA (\eta(0),\psi(0))\rA_{L^2\times \dot H^\mez}}.
$$
As already mentioned, the ratios in the right-hand side measure the frequency localization of the initial data. 
This shows that, in this case, Theorem~\ref{P:quant} gives a quantitative bound in terms 
of the frequency localization of the initial data.
\end{rema}
\begin{proof}
The existence of a solution follows from classical arguments and we prove only the estimate \e{b346}. 
In doing so, it is convenient to symmetrize this system. Consider the Fourier multiplier 
$\la D_x\ra^{\mez}$ and set $\theta=\la D_x\ra^\mez\psi$, which means that, if 
$\psi=\sum_{n\in \xZ} \psi_n e^{i nx}$, then $\theta=\sum_{n\in\xZ}\sqrt{\la n\ra} \psi_n e^{inx}$. 
The equations can be written under the form
$$
\partial_t u+Lu+Pu=0,
$$
where 
$$
u=\begin{pmatrix} \eta \\ \theta\end{pmatrix},\quad 
L=\begin{pmatrix} 0 & -\la D_x\ra^{\mez} \\ \la D_x\ra^{\mez} & 0\end{pmatrix},\quad 
P=\begin{pmatrix} 0 & 0 \\ 0 & -\px^{-1}\la D_x\ra^{\mez}\big(\chi \px^{-1}\la D_x\ra^{\mez} \cdot\big)\end{pmatrix}.
$$
Denote by $(\cdot,\cdot)$ the scalar product in $L^2(S^1)\times L^2(S^1)$.
We obtain $L^2$ estimates for $u$ 
by a simple integration by parts.
Indeed, since $L=-L^{*}$, we obtain
\be\label{p420}
\frac{d}{dt}
\lA u\rA_{L^{2}}^{2}+(Pu,u)=0.
\ee
Now $(Pu,u)\ge 0$, and hence we have the estimate $\lA u(t)\rA_{L^2}\le \lA u(0)\rA_{L^2}$ for all $t\ge 0$. 

To estimate the Sobolev norms of $u(t)$, we cannot simply commute 
spatial derivatives to the equation. Indeed, since $P$ is an operator with variable coefficients, 
the commutator between $P$ and spatial derivatives 
does not vanish and then using the Duhamel formula we would obtain a bound which is not uniform in $t$. 
To overcome this difficulty, we commute the time derivative $\partial_t$ with the equation. 
Set $\dot u=\partial_t u$. Then 
$\dot u$ solves the same equation, so the previous $L^2$-bound applied with $u$ replaced by $\dot u$ gives 
the estimate
$$
\lA \partial_t u(t)\rA_{L^2}\le \lA \partial_t u(0)\rA_{L^2}.
$$
On the other hand, using the equation \e{p410} and the triangle inequality, we get
\begin{align*}
&\lA Lu(t)\rA_{L^2}\le \lA \partial_t u(t)\rA_{L^2}+\lA P u(t)\rA_{L^2},\\
&\lA \partial_t u(0)\rA_{L^2}\le \lA Lu(0)\rA_{L^2}+\lA Pu(0)\rA_{L^2}.
\end{align*}
By combining the previous estimates with the easy bounds 
\begin{align*}
&\lA u(t)\rA_{H^{1/2}}\le \lA Lu(t) \rA_{L^2}+\lA u(t)\rA_{L^2},\quad &\lA Lu(0)\rA_{L^2}\le \lA u(0)\rA_{H^{1/2}},\\
&\lA P u(t)\rA_{L^2} \le K\lA u(t)\rA_{L^2},\quad & \lA P u(0)\rA_{L^2} \le K\lA u(0)\rA_{L^2},
\end{align*}
we conclude that
$$
\lA u(t)\rA_{H^{1/2}}\le C \lA u(0)\rA_{H^{1/2}},
$$
for some constant $C$ independent of time. Iterating this argument, we obtain 
$\lA u(t)\rA_{H^{k/2}}\le C_k \lA u(0)\rA_{H^{k/2}}$ for any integer $k$.
\end{proof}

\section{Luke's variational principle}\label{Luke}

Our goal in this section is to relate the function $\Theta$ 
with Luke's variational principle. Consider the case $P_{ext}=0$. 
Following Luke, the gravity water-wave system can be derived by minimizing the following Lagrangian: 
$$
\mathcal{L}=\int_{t_0}^{t_1}\iint_{\Omega(t)} p\dydxdt=-
\int_{t_0}^{t_1}\iint_{\Omega(t)}\left( \partial_t \phi+\mez\la\nabla_{x,y}\phi\ra^2+gy\right)\dydxdt.
$$
Now observe that
\be\label{b18a}
\int_{-h}^{\eta}\partial_t \phi\dy=\partial_t \left(\int_{-h}^{\eta} \phi\dy
\right) 
-(\partial_t\eta) \psi,\quad 
\iint_{\Omega} gy\dydx = \int_0^L \frac{g}{2}\eta^2\dx -\frac{gLh^2}{2},
\ee
and recall that $\partial_t\eta=G(\eta)\psi$ and also the fact that the kinetic energy is given by 
$\mez \int \psi G(\eta)\psi \dx$. We thus find that
$$
\mathcal{L}=\int_{t_0}^{t_1}\left(K(t)-P(t)\right)\dt +C,
$$
where $C$ is a constant, depending only on $h,L,t_0,t_1$, which 
does not contribute to a variational principle). The previous identity relates $\mathcal{L}$ to the usual 
expression for the Lagrangian the difference between 
the averaged kinetic energy and the averaged potential energy. 

Now, instead of \e{b18a}, write
$$
\int_{-h}^{\eta}\partial_t \phi\dy=\partial_t \left(\int_{-h}^{\eta} \phi\dy
+\psi \eta\right) 
-\eta\partial_t \psi,
$$
to obtain that the Lagrangian $\mathcal{L}$ can be written under the form 
$$
\mathcal{L}=\mathcal{L}'+C-\int\eta\psi\dx \Big\arrowvert_{t=t_0}^{t=t_1}
$$
where $C$ is as above and
$$
\mathcal{L}'=\int_{t_0}^{t_1}\int \left( -\eta\partial_t\psi-\frac{g}{2}\eta^2-\mez \psi G(\eta)\psi\right)\dxdt.
$$
Now, by definition of $\Theta=-\eta\partial_t\psi-\frac{g}{2}\eta^2$, this gives
$$
\mathcal{L}'
=\int_{t_0}^{t_1}\int \left( \Theta-\mez \psi G(\eta)\psi\right)\dxdt .
$$

\section{Another integral identity}\label{A:Identity}

In this section we prove an integral identity analogous to the one obtained in Theorem~\ref{T1bis}. 
The main difference between these two results is that they involve two different observation terms.

\begin{theo}\label{T1}
Let $m\in C^\infty([0,L])$ with $m(0)=m(L)=0$. Then, for any regular solution of \e{systemG} defined on the time interval $[0,T]$, 
there holds
\be\label{b47}
\begin{aligned}
\mez \int_0^T \Hr(t)\dt+\mathcal{P}&=
-\int_0^T\int_0^L P_{ext}\left(\partial_x(m\eta)-\uq \eta\right)\, \dxdt\\
&\quad +\frac{g}{2}\int_0^T\int_0^L(1-m_x)\eta^2\dxdt\\
&\quad+\mez \int_0^T\iint_{\Omega(t)} (1-m_x)\left(\phi_x^2-\phi_y^2\right)\dydxdt\\
&\quad-\int_0^L \left(\px(m \eta)-\uq \eta\right) \psi\, \dx\ST\\
&\quad +\int_0^T\int_0^L \left(\tdm\eta_x-\mez \px (m_x\eta)\right) \phi_y\, \phi_x \dydxdt,
\end{aligned}
\ee
where
$$
\mathcal{P}= \int_0^T\int_0^L \left( \mez h+\frac {3-m_x}{4}\eta\right)\phi_x^2\arrowvert_{y=-h}\dxdt.
$$
\end{theo}
\begin{proof}
We have already proved (see \e{b47aa}) that
\be\label{b167}
\begin{aligned}
\mez \int \Hr(t)\dt&=-\iint P_{ext}\left(\partial_x(m\eta)-\uq \eta\right)\, \dxdt\\
&\quad +\frac{g}{2}\iint(1-m_x)\eta^2\dxdt\\
&\quad-\int \left(\px(m \eta)-\uq\eta\right) \psi\, \dx\ST\\
&\quad -R_a-R_c+\mez R_b,
\end{aligned}
\ee
where $R_a$ is given by Proposition~$\ref{P1}$, $R_b$ is given by Lemma~$\ref{L3}$Ê
and $R_c$ is given by~\e{b168}. 
Consequently, it remains only to prove that
\be\label{b54}
R_a+R_b-\mez R_c=\mathcal{P}+\mathcal{N}+\mathcal{B}
\ee
where
\begin{align*}
\mathcal{P}&= \iint \left(\mez h+\frac {3-m_x}{4}\eta\right)\phi_x^2\arrowvert_{y=-h}\dxdt,\\[1ex]
\mathcal{N}&= -\iiint \left(\tdm\eta_x-\mez \px (m_x\eta)\right) \phi_y\, \phi_x \dydxdt,\\[1ex]
\mathcal{B}&= \iiint  \frac{m_x-1}{2}\left(\phi_x^2-\phi_y^2\right)\dydxdt.
\end{align*}

\begin{lemm}
Set
$$
V=(\px\phi)\aeta,\quad B=(\py\phi)\aeta.
$$
Then
\be\label{b31}
R_a=\mez \iint \left( (G(\eta)\psi)m V+ B m\psi_x\right)\, \dxdt.
\ee
\end{lemm}
\begin{proof}
Recall from \e{b31b} that
$$
R_a=\iint (G(\eta)\psi)m\psi_x\dxdt+ \iint (N(\eta)\psi)m\eta_x\dxdt.
$$
Now write 
\begin{align*}
&(G(\eta)\psi)m\psi_x+(N(\eta)\psi)m\eta_x\\
&\qquad\qquad=(G(\eta)\psi)\left(m\psi_x-mB\eta_x\right)+\left(\mez V^2+\mez B^2\right)m\eta_x\\
&\qquad\qquad=(G(\eta)\psi)m V+\left(\mez V^2+\mez B^2\right)m\eta_x\\
&\qquad\qquad=\mez (G(\eta)\psi)m V+\left[\left(\mez V^2+\mez B^2\right)m\eta_x+\mez (G(\eta)\psi)m V\right]\\
&\qquad\qquad=\mez (G(\eta)\psi)m V+\left[\left(\mez V^2+\mez B^2\right)m\eta_x+\mez (B-\eta_x V)m V\right]\\
&\qquad\qquad=\mez (G(\eta)\psi)m V+\left[\mez B^2m\eta_x+\mez B m V\right]\\
&\qquad\qquad=\mez (G(\eta)\psi)m V+\mez Bm\psi_x,
\end{align*}
which implies that $R_a$ can be written under the form \e{b31}.
\end{proof}

Next, we express $R_a,R_b$ and $R_c$ in terms of integrals of $\nabla_{x,y}\phi$.
\begin{lemm}
There holds
\begin{align}
R_a&= \iiint \frac{m_x}{2}\left((\px\phi)^2-(\py\phi)^2\right)\dydxdt, \label{b48a} \\ 
R_b&=\uq \iiint \left( (\px\phi)^2-(\py\phi)^2\right)\dydxdt-\frac h 4 \iint (\px\phi)^2\arrowvert_{y=-h}\dxdt ,\label{b48b} \\
R_c&= \mez \iiint \left[m_x\left( (\px\phi)^2-(\py\phi)^2\right)-2m_{xx}y (\px\phi)(\py\phi)\right]\dydxdt \notag\\
&\quad -\frac h 2 \iint m_x(\px\phi)^2\arrowvert_{y=-h}\dxdt\label{b48c}.
\end{align}
\end{lemm}
\begin{proof}
To obtain these identities, we will write $R_a,R_b$ and $R_c$ under the form 
$$
\iint u(t,x,\eta(t,x))\dxdt+\iint f(t,x,\eta(t,x))\eta_x(t,x)\dxdt
$$
and then apply the rule \e{b41ab} whose statement is recalled here: for any functions $u=u(x,y)$ and $f=f(x,y)$ with $f\arrowvert_{x=0}=f\arrowvert_{x=L}=0$, one has
\be\label{b41bis}
\int u(x,\eta(x))\dx+\int f(x,\eta)\eta_x\dx
\\=\iint(\py u-\px f)\, \dy\dx
+\int u(x,-h)\dx.
\ee

\smallbreak
\noindent{\em Computation of $R_a$.} Recall that 
$$
R_a=\mez \iint \left[  (G(\eta)\psi)m V+ B m\psi_x\right] \dxdt.
$$
By definition one has
$$
G(\eta)\psi=(\py \phi-\eta_x\px \phi)\aeta,\quad V=(\px\phi)\aeta,\quad 
B=(\py\phi)\aeta,
$$
so
$$
\mez \int \left[  (G(\eta)\psi)m V+ B m\psi_x\right] \dx=\int u(x,\eta)\dx +\int f(x,\eta)\eta_x\dx
$$
with
$$
u(x,y)=m(\px\phi)(\py\phi),\qquad
f(x,y)=\frac{m}{2}\left( (\px\phi)^2-(\py\phi)^2\right).
$$
Since $f\arrowvert_{x=0}=f\arrowvert_{x=L}=0$ and $u\arrowvert_{y=-h}=0$, it follows from \e{b41bis} that
$$
\mez \int \left[  (G(\eta)\psi)m V+ B m\psi_x\right] \dx=\iint(\py u-\px f)\, \dy\dx.
$$
Now, using that $\phi$ solves $\px^2\phi+\py^2\phi=0$, we easily find that 
$$
\py u-\px f=\frac{m_x}{2}\left((\px\phi)^2-(\py\phi)^2\right),
$$
so we verify the identity \e{b48a} for $R_a$.

\smallbreak
\noindent{\em Computation of $R_b$.} 
We have to compute
$$
\int u(x,\eta)\dx +\int f(x,\eta)\eta_x\dx
$$
with
$$
u(x,y)=\uq y\left[ (\px\phi)^2-(\py\phi)^2\right],\qquad
f(x,y)=\mez y (\px\phi)(\py\phi).
$$
Since $m(0)=m(L)=0$, one has $f\arrowvert_{x=0}=f\arrowvert_{x=L}=0$ and hence 
the wanted identity for $R_b$ follows from \e{b41bis}.

\smallbreak
\noindent{\em Computation of $R_c$.} It remains only to compute 
$$
R_c=\iint m_x\eta\left(\mez V^2-\mez B^2+BV \eta_x\right)\dxdt.
$$
As above we have
$$
\int m_x\eta\left(\mez V^2-\mez B^2+BV \eta_x\right)\dx=
\int u(x,\eta)\dx +\int f(x,\eta)\eta_x\dx
$$
with
$$
u(x,y)=\mez m_x y\left((\px\phi)^2-(\py\phi)^2\right),\qquad
f(x,y)=m_x y (\px\phi)(\py\phi).
$$
Since $\px\phi$ vanishes for $x=L$, we have $f\arrowvert_{x=0}=f\arrowvert_{x=L}=0$. On the other hand, one has
\begin{align*}
&u\arrowvert_{y=-h}=\mez m_x (\px\phi)^2\arrowvert_{y=-h},\\ 
&\py u-\px f=\mez m_x\left[ (\px\phi)^2-(\py\phi)^2\right]-m_{xx}y(\px\phi)(\py\phi),
\end{align*}
so \e{b48c} follows from \e{b41bis}.
\end{proof}

\begin{lemm}
There holds
\begin{align}
&\iint \rho(x)\left( \phi_x^2-\phi_y^2\right)\dydx -h\int \rho(x)\phi_x^2\arrowvert_{y=-h}\dx\label{b51}\\
&\qquad =\int \rho \eta \phi_x^2(x,-h)\, dx-2\iint \rho \eta_x \phi_y\, \phi_x \,dydx+2\iint \rho_x(y- \eta) \phi_y \phi_x\dydx.\notag
\end{align}
\end{lemm}
\begin{proof}
Set, for some fixed $t$,
$$
u(x,y)=-\rho(x)(y-\eta(t,x))(\py\phi)(t,x,y)^2.
$$
Then $u(x,\eta(t,x))=0$ and $u(x,-h)=0$ and hence $\int_{-h}^{\eta(t,x)} u_y\, dy=0$. 
On the other hand
$$
u_y=-2\rho(y-\eta)\phi_y \phi_{yy}-\rho(\phi_y)^2,
$$
so integrating on $y\in [-h,\eta(x)]$ and then on $x\in [0,L]$ we obtain, 
remembering that $\phi_{yy}=-\phi_{xx}$,
$$
0=\iint u_y=-\iint \rho \phi_y^2 +2\iint \rho(y-\eta)(\py\phi)(\px^2\phi).
$$
Now set $v\defn  \rho(y-\eta)(\py\phi)(\px\phi)$ and write 
\begin{align*}
&\iint \rho(y-\eta)(\py\phi)(\px^2\phi)\dydx=\iint \px v \dydx\\ 
&\qquad=+\iint \left\{-\rho_x(y-\eta)(\py\phi)(\px\phi)+\rho\eta_x (\py\phi)(\px\phi)-\rho(y-\eta)(\py\px\phi)(\px\phi)\right\}\dydx.
\end{align*}
Observe that $\iint \px v\dydx=0$ since $\int v\arrowvert_{x=0,L}\, dx=0$ and since $v\arrowvert_{y=\eta}=0$. We deduce that
$$
0=-\iint \rho \phi_y^2 -2\iint \rho (y-\eta)(\py\px\phi)(\px\phi)+2\iint \rho \eta_x \phi_y \phi_x-2\iint \rho_x(y- \eta) \phi_y \phi_x,
$$
so
\begin{align*}
0&=-\iint \rho\phi_y^2 -\iint \py\big( \rho (y-\eta)\phi_x^2\big)+\iint \rho \phi_x^2 \\
&\quad +2\iint \rho\eta_x \phi_y \, \phi_x-2\iint \rho_x(y- \eta) \phi_y \phi_x,
\end{align*}
and hence
$$
0=\iint \rho(\phi_x^2-\phi_y^2)-\int (h+\eta)\rho \phi_x^2(x,-h)\, dx+2\iint \rho \eta_x \phi_y\, \phi_x-2\iint \rho_x(y- \eta) \phi_y \phi_x,
$$
which concludes the proof.
\end{proof}
We are now in position to obtain \e{b54} which will conclude the proof of the theorem. Firstly, we write
$$
R_a= \iiint \mez \left(\phi_x^2-\phi_y^2\right)\dydxdt+\iiint \frac{m_x-1}{2}\left(\phi_x^2-\phi_y^2\right)\dydxdt,
$$
to obtain that, using \e{b51} with $\rho=1$,
\be\label{b57}
\begin{aligned}
R_a&=\mez\iint (h+\eta)\phi_x^2\arrowvert_{y=-h}\dxdt\\
&\quad -\iiint  \eta_x \phi_y\, \phi_x \dydxdt+\iiint \frac{m_x-1}{2}\left(\phi_x^2-\phi_y^2\right)\dydxdt.
\end{aligned}
\ee
Directly from \e{b51} applied with either $\rho=1$ or $\rho=m_x$, we find that
\be\label{b58}
R_b=\uq\iint \eta \phi_x^2(x,-h)\dxdt-\mez\iiint  \eta_x \phi_y\, \phi_x \dydxdt,
\ee
and
\begin{align*}
R_c&= - \iiint m_{xx}\, y \, \phi_x \, \phi_y \dydxdt \\
&\quad +\mez \iint m_x \,\eta\, \phi_x^2(x,-h)\dxdt-\iiint m_x \,\eta_x \,\phi_y\, \phi_x \dydxdt\\
&\quad+\iiint m_{xx}\,(y- \eta) \phi_y \,\phi_x\dydxdt.
\end{align*}
which simplifies to 
\be\label{b59}
R_c= +\mez \iint m_x \,\eta\, \phi_x^2(x,-h)\dxdt-\iiint \px(m_x \,\eta) \,\phi_y\, \phi_x \dydxdt.
\ee
We have proved \e{b54} which concludes the proof of Theorem~\ref{T1}.
\end{proof}

\vspace{3cm}

\noindent Thomas Alazard\\[1ex]
CNRS et D\'epartement de Math\'ematiques et Applications UMR 8553\\
\'Ecole normale sup\'erieure \\
45 rue d'Ulm\\
Paris F-75005, France

\end{document}